\documentclass[12pt]{article}
\usepackage[dvipsnames]{xcolor}
\usepackage{amsfonts,amsmath,amssymb,latexsym,calrsfs,amsthm,graphicx}
\usepackage{epsfig,subfigure, rotating, pstricks,graphicx,enumerate,amsmath, tikz,pgfplots}
\usepackage[scr]{rsfso}

\newcommand\redout{\bgroup\markoverwith
	{\textcolor{red}{\rule[0.5ex]{2pt}{0.8pt}}}\ULon}
	
\definecolor{auburn}{rgb}{0.43, 0.21, 0.1}
\definecolor{atomictangerine}{rgb}{1.0, 0.6, 0.4}
\newtheorem{theorem}{Theorem}[section]

\newtheorem{definition}[theorem]{Definition}
\newtheorem{lemma}[theorem]{Lemma}
\newtheorem{prop}{Proposition}

\newtheorem{remark}[theorem]{Remark}

\newtheorem{example}[theorem]{Example}
\newtheorem{thm}{Theorem}[section]

\DeclareMathOperator{\hess}{Hess}
\DeclareMathOperator{\cl}{cl}
\DeclareMathOperator{\id}{Id}
\DeclareMathOperator{\intt}{R_b}

\DeclareMathOperator{\cat}{cat}
\DeclareMathOperator{\dist}{dist}
\DeclareMathOperator{\codim}{codim}

\title{$C^1$ perturbations of a continuum of critical points}
\author{R. Ortega and A.J. Ure\~na}

\begin{document}
\maketitle

\begin{abstract}
	Given a real-valued function having a nondegenerate compact manifold of critical points, some of these points survive under small $C^2$-perturbations. This is a well-known result in critical point theory. In \cite{wei2}, Weinstein obtained the analogous conclusions when the perturbation is only $C^1$ and the ambient space is a finite-dimensional manifold. In this work we present a complete proof for $C^1$ perturbations in infinite-dimensional Hilbert spaces. 
\end{abstract}

\section{Introduction}
Let $\Phi_*=\Phi_*(x)$ be a real valued function of class $C^2$, defined for $x$ on some open set $\Omega$ of a Hilbert space $X$. Assume that $\Omega$ contains a closed manifold\footnote{$M$ is compact, connected, and has no boundary.} $M$ composed of critical points; i.e., \begin{equation*}\nabla\Phi_*(m)=0\qquad\text{for every $m\in M$}.\tag*{{\bf [C]}}
\end{equation*}
Differentiation gives
\begin{equation*}
	T_mM\subset\ker\Big[(\hess\Phi_*)(m)\Big],
\end{equation*}
where $T_m M$ denotes the tangent space at $m\in M$ and $(\hess\Phi_*)(m)$ stands for the second derivative, interpreted as a self-adjoint linear operator. In what follows it will be assumed that the critical manifold $M$ is {\em nondegenerate}, meaning that
\begin{equation*}
	T_mM=\ker\Big[(\hess\Phi_*)(m)\Big]\tag*{\bf{[ND]}}\qquad\text{for every }m\in M\,.
\end{equation*}
 The main issue in this paper is the persistence of some critical points after small perturbations of $\Phi_*$. In other words, if $\Phi:\Omega\to\mathbb R$ is close enough to $\Phi_*$, can one ensure that $\Phi$ has some critical points near $M$?

\medbreak

The origins of this problem in the abstract Calculus of Variations can be traced back to the early work by Krasnosel'skiĭ in the fifties \cite{kra00}, \cite[Theorem 2.2, p. 332]{kra0}. He studied nonlinear eigenvalue problems of the form 
\begin{equation*}
(\nabla\Psi) z=\lambda z\,,\qquad z\in Z\,,
\end{equation*} 
where $Z$ is some Hilbert space and the functional $\Psi:Z\to\mathbb R$ has class $C^2$. Assuming that $z=0$ is a critical point of $\Psi$ then $(\lambda,z)=(0,0)$ is a trivial solution for every $\lambda\in\mathbb R$, and we are interested in bifurcations from this continuum. A well-known necessary condition for the pair $(\lambda_0,0)$ to be a bifurcation point of is that $\lambda_0$ belongs to the spectrum of the Hesssian operator $\hess\Psi(0):Z\to Z$. In this case, and under some compactness assumptions on $\nabla\Psi$, one can use a blowup argument and rescale variables near $(\lambda_0,0)$ to obtain family of functionals $\Phi_\epsilon\to\Phi_*$ such that: {\em (a).} for $\epsilon>0$ the critical points of $\Phi_\epsilon$ give rise to nontrivial solutions; {\em (b).} $\Phi_*$ has a finite-dimensional sphere of critical points.

\medbreak

A second motivation for our problem comes from the theory of Hamiltonian systems. Assuming that $H=H(q,p)$ is the Hamiltonian function associated to the system
\begin{equation*}
\dot q=\frac{\partial H}{\partial p}(q,p)\,,\qquad\dot p=-\frac{\partial H}{\partial p}(q,p)\,,\qquad p,q\in\mathbb R^N\,,
\end{equation*}
many invariant manifolds can be found when $H$ is integrable in Liouville sense. It is well-known that the phase space is foliated by invariant tori, and in superintegrable cases other invariant manifolds with more intricate topology can appear. In some cases such as the Kepler problem one may even find continuos families of closed orbits, and a classical problem going back to Poincar\'e is the existence of closed orbits after small perturbations of the system. The connection between this question and our initial abstract problem comes from the well-known fact that the search of closed orbits of Hamiltonian systems admits a variational formulation on the fractional Sobolev space $H^{1/2}(\mathbb R/\mathbb Z,\mathbb R^{2N})$. Studies in this direction go back to the seventies and include the works of Weinstein \cite{wei}, Moser \cite{mos} or Bartsch \cite{bar}. The conclusion was that the number of closed orbits of the perturbed system is linked to the topology of the invariant manifold via Lusternik-Schnirelmann theory. 

\medbreak

Going back to the abstract setting of arbitrary Hilbert spaces and functionals, the persistence of critical points after perturbation has been analyzed by several authors, including Reeken \cite{ree1, ree2} Ambrosetti, Coti-Zelati and Ekeland \cite{ACE}, Dancer \cite{dan}, Vanderbauwhede \cite{van}, Chillingworth \cite{chi}, or Chang \cite{cha0},\ \cite[Theorem 6.4, p. 135]{cha1}. The succinct argument by Chang requires the perturbed functional $\Phi$ be close to the unperturbed one $\Phi_*$ only in the $C^1$ sense, but assumes finite index of $\hess\Phi_*$ at the critical points, thus excluding the case of Hamiltonian systems. Chang's finite index condition is not needed in this paper.

\medbreak

On the other hand, Szulkin \cite{szu, szu2} and, very recently Zhao \cite{zha}, have showed the existence of closed orbits for $C^1$-small perturbations of some Hamiltonian systems. We will follow along the lines of \cite{szu}. The proof in \cite{zha} is obtained by a different technique based on Floer homology. See \cite{fravan} for the details.
 
\medbreak

The analysis of the simplest possible case already sheds some light. Assuming that $M=\{m_0\}$ is a singleton, condition {\bf [ND]} states that $m_0$ is a nondegenerate critical point of $\Phi_*$, and one expects $\Phi$ to have at least one nearby critical point. When $\Phi$ and $\Phi_*$ are close in the $C^2$ sense, this is indeed a direct consequence of the implicit function theorem. But if $\Phi$ and $\Phi_*$ are close only in the $C^1$ sense the question is subtler. If $X$ is finite-dimensional the proof follows from topological degree theory, but for $\dim(X)=+\infty$ a remarkable example by Reeken \cite[Example 4.2]{ree2} shows that $\Phi$ may have no critical points\footnote{In his discussion Reeken writes that $F$ has compact support but this is a misprint. The function with compact support is $F'$.}. Some additional compactness assumptions on $\Phi$ are thus required.

\medbreak

We close this section by quoting Weinstein (\cite[p. 273]{wei2}): 
\begin{flushleft}
	{\em I would like to stress that the theorem holds only for finite-dimensional manifolds, and that the lack of an infinite-dimensional version has been an obstacle to extending the Conley-Zehnder theorem to manifolds other than tori.}	
\end{flushleft}

\medbreak

The remaining of this paper is organized as follows. Section \ref{sec2} is devoted to present Theorem \ref{th1}, which is our main result. We also include an example showing that the result does not extend to $C^0$ perturbations. Subsequently the theorem is proved through sections \ref{seci}-\ref{secf}. More precisely, sections \ref{seci}-\ref{sec41a} deal with the linear framework. Section \ref{sec5} studies the geometry of the functional near the manifold. In Section \ref{sec6.1} we recall a variant of the Lusternik-Schnirelmann category in infinite dimensions going back to Szulkin \cite{szu}. Estimating the category requires finite-dimensional approximations and this work is carried out in sections \ref{sec7.1}-\ref{secf}.
In order to streamline the presentation we have selected five auxiliary results, labeled as propositions  \ref{lem4.1}-\ref{corro}, all which play a role in the proof but whose proofs are postponed to the Appendix (Section \ref{App}), at the end of the paper.

\section{The main result}\label{sec2}

 In what follows $X$ stands for a separable Hilbert space and the linear operator $L:X\to X$ is continuous, selfadjoint and Fredholm. Both $X$ and $L$ will be fixed throughout the paper. We recall that $L$ being Fredholm means that 
 $$\dim(\ker L)<\infty,\qquad L(X)\text{ closed in $X$},\qquad \codim(L(X))<+\infty\,.$$ Since $L$ has further been assumed selfadjoint, then $\dim(\ker L)=\codim(L(X))$, i.e., $L$ must actually be Fredholm {\em of index zero}.
 
 \medbreak
 
 Assume now that $\Omega\subset X$ is an open subset, which will also be fixed throughout the paper. Our concern is the study of critical points of functionals $\Phi:\Omega\to\mathbb R$ of the form
 \begin{equation}\label{e2}
\Phi(x):=\frac{1}{2}\langle Lx,x\rangle+\Psi(x),\qquad x\in\Omega,
\end{equation} 
for some function $\Psi\in C^1(\Omega)$ {\em with compact gradient}. Thus, the functionals under consideration result from the addition of the quadratic form associated to $L$ plus a remaining term $\Psi$ whose gradient carries closed bounded subsets of $X$ contained in $\Omega$ into relatively compact sets. This class of functionals has a long standing tradition, see \cite{benrab},\cite{conzeh},\cite[Chapter 6]{rab}, \cite[Section 4]{szu} or \cite[Section 2]{szu2}. 
\medbreak

In order to state our main multiplicity result we recall that to every compact manifold $M$ one can associate a topologically invariant integer called its {\em $\mathbb Z_2$-cuplength}\footnote{We use $\mathbb Z_2$ coefficients in order to avoid taking care of orientability issues.} and denoted by $\cl(M)$ (we shall ignore the $\mathbb Z_2$-affix in what follows). It satisfies $1\leq\cl(M)\leq\dim(M)$ and roughly speaking, measures the topological complexity of $M$. For instance, the cuplength of the $N$-sphere $\mathbb S^N$ is $1$ if $N\geq 1$, the cuplength of the $N$-torus is $N$, and $\cl(M_1\times M_2)\geq\cl(M_1)+\cl(M_2)$ for any compact manifolds $M_1,M_2$ (\cite[Lemma 5.15, p. 161]{sch}). The cuplength is related to the Lusternik-Schnirelmann category by the inequality:
\begin{equation}\label{e4}
\cat(M)\geq\cl(M)+1,
\end{equation} 
see, e.g., \cite[Corollary, p. 161]{sch}. In particular, every continuously-differentiable function $f:M\to\mathbb R$ always has at least $\cl(M)+1$ critical points. The main result of these notes is the following:

\begin{thm}\label{th1}{Let the unperturbed functional
		$\Phi_*:\Omega\to\mathbb R$ be as in \eqref{e2} for some function $\Psi_*\in C^2(\Omega)$ with compact gradient. If there exists a finite-dimensional, compact, $C^1$-submanifold without boundary $M\subset\Omega$ with {\bf [C]}-{\bf [ND]}, then there is some $\epsilon>0$ such that, whenever $\Psi\in C^1(\Omega)$ has compact gradient  and satisfies
		\begin{equation}\label{e1}
		\|\nabla\Psi(x)-\nabla\Psi_*(x)\|<\epsilon\ \forall x\in\Omega,		
		\end{equation}
			then the corresponding functional $\Phi$ (defined as in \eqref{e2}),
		has at least $1+\cl(M)$ different critical points in $\Omega$.}
\end{thm}

Observe that the condition of $\Phi$ having the form \eqref{e2} is relevant only when $\dim X=+\infty$, for it is clear that {\em in the finite-dimensional situation all continuously-differentiable functionals $\Phi$ can be decomposed in this way}.

\medbreak

Theorem \ref{th1} suggests that one could perhaps replace condition \eqref{e1} by the assumption that $\Phi$ is close to $\Phi_*$ in the uniform topology. Can one still show the existence of critical points of $\Phi$ near $M$? Assuming for instance that $X$ is finite-dimensional and $\hess\Phi_*(m)$ is positive definite on $(T_mM)^\bot$ for every $m\in M$, then for any sufficiently-small neighborhood $B$ of $M$ one will have $\inf_{B}\Phi_*<\inf_{\partial B}\Phi_*$; this inequality will be inherited by the perturbed functional $\Phi$ and one can expect that direct minimization methods will imply the existence of at least one critical point (minimum) of $\Phi$ in $B$. However, if $\hess\Phi_*$ is indefinite at the points of $M$ the result fails to be true, as the following example (with $X=\mathbb R^2$ and $M=\{(0,0)\}$) shows:
\begin{example}{\rm For each $\epsilon>0$ we consider the continuously-differentiable function $\Phi_\epsilon:\mathbb R^2\to\mathbb R$ defined by:
$$\Phi_\epsilon(x,y):=x^2-y^2+\sqrt[4]{\epsilon}\,f_\epsilon(x+y)\,,$$
where $f_\epsilon\in C^1(\mathbb R,\mathbb R)$ is the function given by
$$f_\epsilon(z):=\begin{cases}
	\sqrt{\epsilon+z}-\sqrt{\epsilon}\,,&\text{if }z\geq 0\,,\\
	-f_\epsilon(-z)\,,&\text{if }z< 0\,.
\end{cases}$$
It is clear that $\lim_{\epsilon\searrow 0}\Phi_\epsilon(x,y)=x^2-y^2$, uniformly on compact subsets of $\mathbb R^2$. However, $\Phi_\epsilon$ {\em does not have critical points on the unit ball $B_1\subset\mathbb R^2$ provided that $0<\epsilon<1/64$}. In fact, a direct computation shows that the unique critical point of $\Phi_\epsilon$ is $(x_\epsilon,y_\epsilon)=\lambda_\epsilon(-1,1)$, with $\lambda_\epsilon=1/(4\sqrt[4]{\epsilon})$.}

\end{example}

\section{Preparing the proof I: some linear functional analysis.}\label{seci} We start with an observation: if $X^-,X_0,X^+$ are closed subspaces of the Hilbert space $X$ such that $\dim X^0<\infty$ and 
\begin{equation}\label{sptY}
X=X^-\oplus X^0\oplus X^+\ \text{ (orthogonal sum),}
\end{equation} then the linear operator 
$L:X\to X$ defined by 
 \begin{equation}\label{eu12}
	L(x)=x^+-x^-,\ \text{ for all } x=x^-+x^0+x^+\in X\,,
\end{equation} is continuous, selfadjoint and Fredholm. Furthermore, 
$$L\text{ vanishes on }X^0,\qquad\text{$L$ negative def. on $X^-$},\qquad\text{$L$ positive def. on }X^+\,.$$ The proposition below states that splittings of this type are a general feature of  selfadjoint Fredholm operators.

\begin{prop}\label{lem4.1}{Let $L:X\to X$ be linear, continuous, selfadjoint and Fredholm, and set $X^0:=\ker L$. Then there exist closed $L$-invariant subspaces $X^\pm\subset X$ with \eqref{sptY} satisfying, for some $r>0$, 
\begin{equation}\label{inex}
\langle Lx^-,x^-\rangle\leq  -r\|x^-\|^2\ \forall x^-\in X^-,\qquad
\langle Lx^+,x^+\rangle\geq r\|x^+\|^2\ \forall x^+\in X^+.
\end{equation}
Furthermore, these subspaces are uniquely determined by $L$.} 	
\end{prop}
Proposition \ref{lem4.1} is the first of the five auxiliary results whose proof is postponed to the Appendix (Section \ref{App}), at the end of the paper.

\medbreak

 	We observe that the condition of $L$ being Fredholm cannot be skipped. For instance, the continuous linear operator
	$$\mathscr T:\mathscr L^2(0,1)\to\mathscr L^2(0,1),\qquad x=x(t)\mapsto tx(t)\,,$$
		is selfadjoint, but not Fredholm. We point out that $\langle\mathscr T x,x\rangle>0$ if $x\not=0$, and in particular, $\ker\mathcal T=\{0\}$. However, the image of $\mathscr T$ is dense in $\mathscr L^2(0,1)$, and there is no $r>0$ with $\langle\mathscr Tx,x\rangle\geq r\|x\|^2$ for every $x\in\mathscr L^2(0,1)$.

\medbreak

The  linear, continuous, selfadjoint and Fredholm operator $L:X\to X$ will be fixed in what follows. Let $\mathcal K(X)$ stand for the Banach space of linear, compact and selfadjoint operators $K:X\to X$. By a theorem due to Weyl  (see, e.g. \cite[p. 367]{riebel}), compact perturbations of Fredholm operators remain Fredholm. Thus,  for every $K\in\mathcal K(X)$ the linear operator $L+K$ inherits from $L$ being continuous, selfadjoint and Fredholm, and proposition \ref{lem4.1} ensures the existence of closed, $L+K$-invariant subspaces $X^\pm_K\subset X$ such that
\begin{itemize}\item $L+K$ is negative-definite on  $X^-_K$ and positive-definite on $X^+_K$, and  \item the Hilbert space $X$ splits as in \eqref{sptY} for $X^0=X^0_K:=\ker(L+K)$.
\end{itemize}
This procedure defines maps 
$$X^-=X^-_K\,,\qquad X^0=X^0_K\,,\qquad X^+=X^+_K\,,$$ from $\mathcal K(X)$ into the Grassmannian $\mathcal G(X)$ of closed, linear subspaces $V\subset X$. 

\medbreak

The set $\mathcal G(X)$ is naturally endowed with a metric space structure defined as follows:
$$d(V_1,V_2):=\|\Pi_{V_2}-\Pi_{V_1}\|_{\mathscr L(X)}\,,$$
(here and in what follows we denote $\Pi_V:X\to V$ the continuous orthogonal projection on the subspace $V\in\mathcal G(X)$). This distance has some interesting properties. For instance, one easily checks that if $V_1,V_2,W_1,W_2$ are closed subspaces of $X$ with  $V_1\bot V_2$ and $W_1\bot W_2$ then
\begin{equation}\label{oplus}
	d(V_1\oplus V_2,W_1\oplus W_2)\leq d(V_1,W_1)+d(V_2,W_2)\,.
\end{equation}
 In addition, if $d(V,W)<1$ then the restriction of the orthogonal projection $\Pi_{W}$ to $V$ defines a topological isomorphism $V\cong W$. See, e.g., \cite[Theorem 1]{buc}.

\medbreak

The question concerning the continuity of the maps $X^\pm,X^0:\mathcal K(X)\to\mathcal G(X)$ arises. The answer is negative in general as, for instance, $X^0_K$ may abruptly increase its dimension if, say, some eigenvalues of $L+K$ suddenly vanish at some operator $K=K_*$. We must therefore be more careful and consider the subsets  
$$\mathcal K_p(X):=\{K\in\mathcal K(X):\dim(X_K^0)=p\}\subset\mathcal K(X)\,,\qquad p\geq 0\,.$$
\begin{prop}\label{lem501}{For each $p\geq 0$, the restrictions of $X^\pm, X^0$ to $\mathcal K_p(X)$ 
		are continuous.}	
\end{prop}
This proposition is the second auxiliary result whose proof is postponed to the appendix (Section \ref{App}), at the end of the paper. 

\section{Preparing the proof II: linear saddles at each point of $M$}\label{sec41a}
Using an idea going back to \cite[Remark 1.10]{bencapfor}, in the remaining of this paper we shall assume, without loss of generality, that the operator $L$ has the form \eqref{eu12}.
Indeed, in the general case it suffices to replace the scalar product $\langle\cdot,\cdot\rangle$ of $X$ by the equivalent one
$$(x_1|x_2):=\langle L x_1^+,x_2^+\rangle+\langle x_1^0,x_2^0\rangle-\langle Lx_1^-,x_2^-\rangle,\qquad x_i=x_i^-+x_i^0+x_i^+\in X\,.$$

From now on we fix, once and for all, the functionals $\Psi_*,\Phi_*:\Omega\to\mathbb R$, as well as the nondegenerate critical manifold $M\subset\Omega$ satisfying the conditions of Theorem \ref{th1}.  Having changed the scalar product, the notion of gradient changes. In consequence, there is an issue on the complete continuity of the gradient of $\Psi_*$ or $\Psi$ with respect to the new scalar product. This is answered by the following remark:

\begin{remark}{\rm $X$ being a Hilbert space, its scalar product defines a (canonical) topological isomorphism between $X$ and its topological dual space $X^*$. For any continuously-differentiable functional $\Psi:\Omega\to\mathbb R$, its derivative and gradient maps $$\Psi':\Omega\to X^*\,,\qquad \nabla\Psi:\Omega\to X\,,$$ are related by this isomorphism. Topological isomorphisms carry relatively compact sets into relatively compact sets, and therefore, $\nabla\Psi$  is compact if and only if $\Psi'$ shares the same property. As a consequence,  {\em the assumption of $\Psi$ having a compact gradient does not depend on the equivalent scalar product used in $X$.}}
\end{remark} 

\medbreak

Since $\nabla\Psi_*:\Omega\to X$ is compact, at each point $m\in M$ its Hessian operator $\hess\Psi_*(m):X\to X$ is compact  \cite[Theorem 17.1, p. 77]{kraszab} and so $\hess\Phi_*(m)=L+\hess\Psi_*(m)$ is a compact perturbation of $L$. By combining Proposition \ref{lem4.1} with assumption {\bf [ND]} we obtain, for each $m\in M$, the splitting
\begin{equation}\label{sptXY}
	X=X^-_m\oplus T_mM\oplus X^+_m\ \text{ (orthogonal sum),}
\end{equation}
the closed subspaces $X^\pm_m$ being $\hess\Phi_*(m)$-invariant, with $\hess\Phi_*(m)$  negative-definite on $X_m^-$ and positive-definite on $X_m^+$. We shall assume, with no loss of generality, that
\begin{equation}\label{n0z}
X_m^-\not=\{0\}, m\in M\,,
\end{equation}
as otherwise it suffices to change the sign of $L$, $\Psi_*$, $\Phi_*$, $\Psi$, $\Phi$.

\medbreak
 
In a first approach we might try to use \eqref{sptXY} as a reference system  to construct a tubular neighborhood of $M$ reflecting the saddle-like geometry of $\Phi_*$ near $M$. Setting $d:=\dim M$, the map  $\hess\Phi_*=L+\hess\Psi_*$ is continuous when seen from  $M$ into $\mathcal K_d(X)$. Thus, Proposition \ref{lem501} implies the continuity of the maps $M\to\mathcal G(X)$, $m\mapsto X_m^\pm$. We deduce that the sets
$$\mathfrak X^\pm:=\{(m,v):m\in M,\ v\in X_m^\pm\}\,,\qquad TM:=\{(m,v):m\in M,\ v\in T_mM\}\,,$$
are {\em topological} vector subbundles of $M\times X$. This smoothness is not enough for our purposes, and we need to approximate these subbundles by differentiable counterparts. It motivates the following:

\begin{definition}[Class {\text [$\mathcal C^1_d$]}] A continuously-differentiable map $$K:M\to\mathcal K(X),\qquad m\mapsto K_m\,,$$ is said to belong to the class $[\mathcal C^1_d]$ provided that
	\begin{equation}\label{cK4432}
		K_m\in\mathcal K_d(X)\,,\qquad \text{ for every }m\in M\,.	
	\end{equation}
\end{definition}

While $K_*:=\hess\Psi_*$ satisfies condition \eqref{cK4432}, it may fail to be continuously-differentiable and so it does not necessarily belong to the class $[\mathcal C^1_d]$. However, it can be  uniformly approximated by $[\mathcal C^1_d]$-maps, as a consequence of the regularization result below:
\begin{prop}\label{prop3}{For every continuous map $K:M\to\mathcal K_d(X)$ there exists a sequence 
		$$K^{(n)}:M\to\mathcal K(M)\,,\qquad m\mapsto K_m^{(n)}\,,$$
		of maps in the class $[\mathcal C^1_d]$, such that 
		$$K_m^{(n)}\to K_m\text{ as } n\to\infty,\text{ uniformly with respect to }m\in M.$$ 
	}
\end{prop}
Given $K:M\to\mathcal K(X)$ we consider the sets $\mathfrak X_K^0,\, \mathfrak X_K^\pm\subset M\times X$ defined by
\begin{equation}
	\label{eu14}\mathfrak X^0_K:=\{(m,v):m\in M,\ v\in X_{K_m}^0\}\,,\quad\mathfrak X^\pm_K:=\{(m,v):m\in M,\ v\in X_{K_m}^\pm\}\,,
	\end{equation}
and seems reasonable to guess that  they inherit some smoothness provided that $K$ belongs to the class $[\mathcal C^1_d]$ . This is indeed the content of the following:

\begin{prop}\label{lem212}{If $K$ belongs to the class $[\mathcal C^1_d]$ then $\mathfrak X_K^0$, $\mathfrak X_{K}^\pm$ are continuously-differentiable vector subbundles of $M\times X$.}
\end{prop}
The proofs of propositions \ref{prop3}-\ref{lem212} are also postponed to the Appendix (Section \ref{App}), at the end of the paper. 
\section{A fibered saddle around the critical manifold}\label{sec5}
Let the map $K:M\to\mathcal K(X)$ belong to the class [$\mathcal C^1_d$], and assume that
\begin{equation}\label{cK}
X^0_{K_m}\cap(T_mM)^\bot=\{0\}\qquad \text{ for every }m\in M,	
\end{equation}
or, equivalently,	
\begin{equation}\label{eq43}
	X=X_{K_m}^-\oplus (T_mM)\oplus X_{K_m}^+\qquad\text{ for every } m\in M\,.
\end{equation}
Recalling Proposition \ref{lem212}, the vector bundles $\mathfrak X^\pm_K$ are continuously-differentiable, and this smoothness is inherited by their Withney sum $$\mathfrak X_K^-\oplus\mathfrak X_K^+:=\Big\{(m,v^-,v^+):m\in M,\ v^\pm\in X_{K_m}^\pm\Big\}\,.$$

Then, the tubular neighborhood theorem states the existence of some constant $r>0$ (not depending on $m$ or $v^\pm$) such that, whenever  $0<r^\pm\leq r$, the set $$B_K(r^-,r^+):=\{m+v^-+v^+:m\in M,\ v^\pm\in X_{K_m}^\pm,\ \|v^\pm\|<r^\pm\}$$
is open in $X$, and the map $(m,v_-,v_+)\mapsto m+v_-+v_+$ defines a continuously-differentiable diffeomorphism from the set of triples $(m,v^-,v^+)\in\mathfrak X_K^-\oplus \mathfrak X_K^+$ with $\|v^\pm\|\leq r^\pm$, to the closure of $B_K(r^-,r^+)$. {\em Under these conditions we shall say that $B:=B_K(r^-,r^+)$ is a $K$-tubular neighborhood of $M$}.  

\medbreak

In this situation, the topological boundary $\partial B$ can be divided into two (nondisjoint) closed sets
$$\partial B=(\partial^-_KB)\cup(\partial^+_KB)\,,$$
where $\partial^\pm_K B:=\{m+v^-+v^+:m\in M,\ v^\pm\in X_{K_m}^\pm,\ \|v^\pm\|=r^\pm\}$. For future reference we also consider the set
$$B^0:=\{m+v^+:m\in M,\ v^+\in X_{K_m}^+,\ \|v^+\|\leq r^+\}\,.$$

\medbreak

Assume now that $\bar B\subset\Omega$ and $\Phi:\Omega\to\mathbb R$ is continuously-differentiable. For each level $c\in\mathbb R$ one can consider the corresponding $\bar B$-sublevel set
$\Phi_c:=\{x\in\bar B:\Phi(x)\leq c\}$. Given $x\in\partial B$ and $\sigma>0$ we  consider the region 
\begin{equation}\label{csig}
	C_\sigma(x):=\{\lambda v:v\in X,\ \|v-\nabla\Phi(x)\|<\sigma,\ 0<\lambda<1\}\,.
\end{equation}
If $\|\Phi(x)\|>\sigma$ this set has the shape of an ice-cream, and it is contained in the convex cone with vertex at the origin generated by the open ball centered at $\nabla\Phi(x)$ and with radius $\sigma$.



	
	

\begin{definition}[Fibered saddle neighborhood of $M$]\label{fstn} The neighborhood $B$ of $M$ with $\dist(B,X\backslash\Omega)>0$ will be said to be of fibered saddle type (relative to $\Phi\in C^1(\Omega)$) provided that  and there exists a map $K:M\to\mathcal K(X)$ in the class $[\mathcal C_d^1]$ and satisfying \eqref{cK} such that $B=B_K(r^-,r^+)$ is a $K$-tubular neighborhood of $M$, and moreover:
	\begin{enumerate}
	\item[(i)] there exists some constant $\sigma>0$ (not depending on $x$), and for every $x\in(\partial^+_K B)\backslash(\partial^-_K B)$ there exists some $\epsilon_x>0$, such that $$x-\epsilon_x\, C_\sigma(x)\subset B,\qquad \big(x+\epsilon_x\, C_\sigma(x)\big)\cap B=\emptyset\,,$$
		\item[(ii)]  $\displaystyle{\sup_{\partial^- B}\Phi<\inf_{B^0}\Phi}$,\ and there exists some  $\displaystyle{c_0\in\left]\sup_{\partial^- B}\Phi,\inf_{B^0}\Phi\right[}$ such that 
		$$\nabla\Phi(x)\not=0\text{ for all }x\in\Phi_{c_0}\,.$$
	\end{enumerate} 
\end{definition}
Condition {\em (i)} can be thought of as a {\em positive outer normal derivative condition} on $(\partial^+B)\backslash(\partial^-B)$. In particular, it forbids  $\Phi$ having critical points on $(\partial^+ B)\backslash(\partial^- B)$. Since {\em (ii)} rules out the presence of critical points on $\partial^- B$, one concludes that $\Phi$ does not have critical points on $\partial B$. 
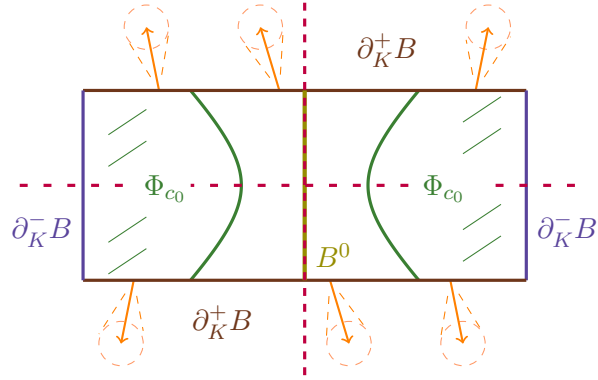
\begin{figure}
\begin{center}
\begin{tikzpicture}
	\begin{axis}[axis equal image, axis lines = none,
		ymin = -3, ymax = 3,
	xmin = -5, xmax = 5,
	width=10cm,
height=7cm,
samples = 100]

\draw[atomictangerine, dashed] (axis cs:2.9,2.5) circle[radius=35];
\draw[atomictangerine, dashed] (axis cs:-2.5,2.5) circle[radius=35];
\draw[atomictangerine, dashed] (axis cs:-0.7,2.5) circle[radius=35];
\draw[atomictangerine, dashed] (axis cs:2.5,-2.5) circle[radius=35];
\draw[atomictangerine, dashed] (axis cs:-2.9,-2.5) circle[radius=35];
\draw[atomictangerine, dashed] (axis cs:0.7,-2.5) circle[radius=35];
\draw[dashed, orange]
(axis cs:-2.3,1.5) -- (axis cs:-2.8,2.3);
\draw[dashed, orange]
(axis cs:-2.3,1.5) -- (axis cs:-2.14,2.45);
\draw[dashed, orange]
(axis cs:-0.4,1.5) -- (axis cs:-1,2.3);
\draw[dashed, orange]
(axis cs:-0.4,1.5) -- (axis cs:-0.34,2.45);
\draw[dashed, orange]
(axis cs:2.7,1.5) -- (axis cs:2.55,2.3);
\draw[dashed, orange]
(axis cs:2.7,1.5) -- (axis cs:3.25,2.35);
\draw[dashed, orange]
(axis cs:2.3,-1.5) -- (axis cs:2.8,-2.3);
\draw[dashed, orange]
(axis cs:2.3,-1.5) -- (axis cs:2.14,-2.45);
\draw[dashed, orange]
(axis cs:0.4,-1.5) -- (axis cs:1,-2.3);
\draw[dashed, orange]
(axis cs:0.4,-1.5) -- (axis cs:0.34,-2.45);
\draw[dashed, orange]
(axis cs:-2.7,-1.5) -- (axis cs:-2.55,-2.3);
\draw[dashed, orange]
(axis cs:-2.7,-1.5) -- (axis cs:-3.25,-2.35);

\fill[white] (axis cs:-2.3,1.5) circle[radius=15];
\fill[white] (axis cs: 2.7,1.5) circle[radius=15]; 
\fill[white] (axis cs:-0.4,1.5) circle[radius=15]; 
\fill[white] (axis cs:2.3,-1.5) circle[radius=15];
\fill[white] (axis cs:-2.7,-1.5) circle[radius=15]; 
\fill[white] (axis cs:0.4,-1.5) circle[radius=15];

\draw[->, thick, orange]
(axis cs:-2.3,1.5)--(axis cs:-2.5,2.5);
\draw[->, thick, orange]
(axis cs:2.7,1.5) -- (axis cs:2.9,2.5);
\draw[->, thick, orange]
(axis cs:-0.4,1.5) -- (axis cs:-0.7,2.5);
\draw[->, thick, orange]
(axis cs:2.3,-1.5) -- (axis cs:2.5,-2.5);
\draw[->, thick, orange]
(axis cs:-2.7,-1.5) -- (axis cs:-2.9,-2.5);
\draw[->, thick, orange]
(axis cs:0.4,-1.5) -- (axis cs:0.7,-2.5); 


	\addplot[
	OliveGreen,
	very thick, domain=-1.5:1.5]
	({sqrt(1 + x^2)},{x});
		\addplot[
	OliveGreen,
	very thick, domain=-1.5:1.5]
	({-sqrt(1 + x^2)},{x});


\draw[OliveGreen](axis cs:-3.1,0.8) -- (axis cs:-2.5,1.2); 
\draw[OliveGreen](axis cs:-3.1,0.3) -- (axis cs:-2.5,0.7); 
\draw[OliveGreen](axis cs:-3.1,-0.9) -- (axis cs:-2.5,-0.5); 
\draw[OliveGreen](axis cs:-3.1,-1.4) -- (axis cs:-2.5,-1); 
\draw[OliveGreen](axis cs:3.1,-0.8) -- (axis cs:2.5,-1.2); 
\draw[OliveGreen](axis cs:3.1,-0.3) -- (axis cs:2.5,-0.7); 
\draw[OliveGreen](axis cs:3.1,0.9) -- (axis cs:2.5,0.5); 
\draw[OliveGreen](axis cs:3.1,1.4) -- (axis cs:2.5,1);


\draw[auburn, very thick] (axis cs:-3.5,1.5) -- (axis cs:3.5,1.5); 
\draw[Violet, very thick](axis cs:3.5,1.5) -- (axis cs:3.5,-1.5); 
\draw[auburn, very thick] (axis cs:3.5,-1.5) -- (axis cs:-3.5,-1.5); 
\draw[Violet, very thick](axis cs:-3.5,-1.5) -- (axis cs:-3.5,1.5);


\draw[olive, ultra thick](axis cs:0,-1.5) -- (axis cs:0,1.5);


\draw[purple, very thick, loosely dashed] (axis cs:-4.5,0) -- (axis cs:4.5,0);
 \draw[purple, very thick, dashed] (axis cs:0,-3) -- (axis cs:0,3);
  \node at (axis cs:0, -1.1) [anchor=west] {\color{olive}\small $B^0$};
  \fill[white] (axis cs:-2.3,0) circle[radius=50];
 \node at (axis cs:-2.2, -0.05) [anchor=center] {\color{OliveGreen}\small $\Phi_{c_0}$};
   \fill[white] (axis cs:2.3,0) circle[radius=50];
   \node at (axis cs:2.2, -0.05) [anchor=center] {\color{OliveGreen}\small $\Phi_{c_0}$};

  
    \node at (axis cs:-3.5, -0.7) [anchor=east] {\color{Violet}\small $\partial^-_KB$};
      \node at (axis cs:3.5, -0.7) [anchor=west] {\color{Violet}\small $\partial^-_KB$};
    \node at (axis cs:1.3, 1.7) [anchor=south] {\color{auburn}\small $\partial^+_KB$};
   \node at (axis cs:-1.3, -2.6) [anchor=south] {\color{auburn}\small $\partial^+_KB$};

   \end{axis}
\end{tikzpicture}
\end{center}
	\caption{The section of $B$ at some value $m\in M$.}
\end{figure}
We point out that the fact that this definition does not presuppose that $M$ is made of critical points of $\Phi$. On the other hand, under the framework of Theorem \ref{th1} we shall construct a tubular neighborhood of $M$ reflecting the geometry of $\Phi_*$ near it. This is the purpose of the following:
\begin{lemma}\label{prop2}{Under the assumptions of Theorem \ref{th1} there exists a neighborhood $B$ of $M$  with $\bar B\subset\Omega$ which is of fibered saddle type relative to $\Phi_*$.}
\begin{proof}
To prove this result we start by applying Proposition \ref{prop3} with the map $K:=\hess\Psi_*$, and we obtain a sequence $K^{(n)}$ in the class $[\mathcal C_d^1]$ uniformly converging to $\hess\Psi_*$ on $M$. In view of Proposition \ref{lem212} the associated vector bundles $\mathfrak X_{K^{(n)}}^\pm$ are continuously differentiable, and it follows from Proposition \ref{lem501} that the associated fibers $X_{K_m^{(n)}}^\pm$ converge to $X_m^\pm$ as $n\to+\infty$, uniformly with respect to $m\in M$. We consider the linear selfadjoint operators $$Q_{n,m}:=\Pi_{X_{\hspace{-0.07cm}K_m^{(n)}}^-}\circ(\hess\Phi_*)\circ\Pi_{X_m^-}\circ\Pi_{X_{\hspace{-0.07cm}K_m^{(n)}}^-}+\Pi_{X_{\hspace{-0.07cm}K_m^{(n)}}^+}\circ(\hess\Phi_*)\circ\Pi_{X_m^+}\circ\Pi_{X_{\hspace{-0.07cm}K_m^{(n)}}^+}\,,$$ 
and we conclude that
\begin{equation}\label{eq:conv87}
\big\|\hess\Phi_*(m)-Q_{n,m}\big\|_{\mathscr L(X)}\to 0\quad\text{as $n\to\infty$, uniformly with respect to $m\in M$.}
\end{equation}

\medbreak

On the other hand $\Phi_*$ is constant on the critical manifold $M$ and so there is no loss of generality in assuming that $\Phi_*(m)=0$ for every $m\in M$. Then, for every $m\in M$ and $x\in X$ with  $\|x\|$ small, Taylor's formula provides the asymptotic expansions
$$\Phi_*(m+x)=\frac{1}{2}\langle\hess\Phi_*(m)x,x\rangle+o(\|x\|^2),\ \ \nabla\Phi_*(m+x)=\hess\Phi_*(m)x+o(\|x\|)\,.$$ 
By compactness these estimates are further uniform with respect to $m\in M$. In combination with \eqref{eq:conv87} we deduce the existence of sequences of positive numbers $r_n,\epsilon_n\searrow 0$ such that, for every $m\in M$, all $n\in\mathbb N$, and every $\tilde x=x^-+x^+\in X_{K_m^{(n)}}^-\oplus X_{K_m^{(n)}}^+$ with $\|\tilde x\|=\sqrt{\|x^-\|^2+\|x^+\|^2}<r_n$ one has
\begin{equation*}
\Big|\Phi_*(m+\tilde x)-\frac{1}{2}\langle Q_{n,m}\,\tilde x,\tilde x\rangle\Big|\leq\epsilon_n\|\tilde x\|^2,\ \ \|\nabla\Phi_*(m+\tilde x)-Q_{n,m}\tilde x\|\leq\epsilon_n\|\tilde x\|\,,	
\end{equation*}
or, what is the same (setting $v^\pm:=\Pi_{X_m^\pm}x^\pm$),
\begin{align*}
&\left|\Phi_*(m+\tilde x)-\frac{1}{2}\left\langle\hess\Phi_*(m)(v^-+v^+),v^-+v^+\right\rangle\right|\leq\epsilon_n\|\tilde x\|^2\,,\\
&\left\|\nabla\Phi_*(m+\tilde x)-\Pi_{X_{\hspace{-0.07cm}K_m^{(n)}}^-}\hess\Phi_*(m)v^--\Pi_{X_{\hspace{-0.07cm}K_m^{(n)}}^+}\hess\Phi_*(m)v^+\right\|\leq\epsilon_n\|\tilde x\|\,.
\end{align*}
At this point we recall that each operator $\hess\Phi_*(m)$ is negative definite on $X_m^+$ and positive definite on $X_m^+$. Moreover, a compactness argument shows that the associated constants can be chosen uniformly with respect to $m\in M$. Then one can take $K:=K^{(n)}$ for $n$ big enough and adjust the small constants $r^\pm>0$ satisfying the statement.
\end{proof}
\end{lemma}

\medbreak

One easily checks that fibered saddle neighborhoods are robust with respect to $C^1$-perturbations of the functional. With other words, if the neighborhood $B$ of $M$ is of fibered saddle type relative to $\Phi_*$, then it will also be of fibered saddle type relative to $\Phi$ provided that the $C^1$-norm of $\Phi-\Phi_*$ is small enough.

\medbreak

It remains to show that under these conditions $\Phi$ has at least $1+\cl(M)$ critical points in $B$. With this goal, a first attempt could consist in constructing `finite dimensional reductions'  of $\Phi$, prove the existence of critical points for these approximations, and try to pass to the limit. However, in this procedure the multiplicity of critical points would likely be lost. Thus, we shall follow a modified strategy consisting in introducing an infinite-dimensional notion of category going back to Szulkin \cite{szu}, observing that the number of critical points can be estimated using this category, and finally give lower bounds on the category based on the cuplength of $M$. The use of finite-dimensional approximations will come at this last step.     

\medbreak

Thus, from now on we fix the functional $\Phi:\Omega\to\mathbb R$ of the form \eqref{e2} for some $\Psi\in C^1(\Omega)$ with compact gradient. In addition we assume that there exists some compact, $C^1$-manifold $M\subset\Omega$ and some neighborhood $B$ of $M$ with $\bar B\subset\Omega$ which is of fibered saddle type relative to $\Phi$.

\section{Szulkin's relative category and critical points}\label{sec6.1}

The Lusternik-Schnirelmann category is a well-known topological tool which has proved useful to study existence and multiplicity of critical points of functionals which are either coercive or defined on a compact manifold. Subsequently it has been generalized in many directions. In order to study functionals which may be unbounded both from above and from below, Fournier and Willem \cite{fouwil} employed the concept of {\em relative category with respect to a subset} (see also Weinstein's memoir \cite[p. 275]{wei2}). Subsequently, the situation where the functional is defined on an infinite-dimensional space was treated by Szulkin. In \cite{szu} he considered the generalized notion of {\em relative category with respect to an admissible class of deformations}, which is very close to the one used below.
\medbreak

As before, we assume that $B\subset\Omega$ is a fibered saddle neighborhood of $M$ relative to the functional $\Phi:\Omega\to\mathbb R$. Pick some constant $c_0\in\mathbb R$ as in Definition \ref{fstn}{\em (ii)}, and let it be fixed in what follows.

\medbreak

Given a closed subset $A\subset\bar B$, by a {\em deformation} of $A$ in $\bar B$ we mean a continuous map $\eta:[0,1]\times A\to\bar B$ such that $\eta(0,a)=a$ for every $a\in A$.  If $A=\bar B$ we shall just say that $\eta$ is a {\em deformation of $\bar B$}. We shall say that $A$ is {\em contractible} in $\bar B$ provided that there exists a deformation $\eta$ of $A$ in $\bar B$ such that $\eta(\{1\}\times A)\subset\bar B$ is a singleton.  We consider the class $\mathcal D^*$ of deformations $\eta$ of $\bar B$ the form
\begin{equation}\label{eq9}
	\eta(t,x)=\exp\big[\theta(t,x)L\big]x+C(t,x)\,,\qquad (t,x)\in[0,1]\times\bar B\,,
\end{equation}
where $\theta:[0,1]\times\bar B\to\mathbb R$ and $C:[0,1]\times\bar B\to X$ are continuous functions satisfying: 
\begin{equation}\label{eq07}
	\begin{cases}
		\theta(0,x)=0,\ C(0,x)=0,\qquad \forall x\in\bar B\,,\\ \\
		\theta([0,1]\times\bar B)\subset\mathbb R\text{ bounded; }\ C([0,1]\times\bar B)\subset X\text{ relatively compact.}
	\end{cases}
\end{equation}

\begin{remark}{\rm Deformations of the form \eqref{eq9} were already considered in \cite[Proposition A.18, p. 86]{rab}. Recalling \eqref{sptY}-\eqref{eu12} one can rewrite \eqref{eq9} more explicitly as follows:
		$$\eta(t,x)=e^{-\theta(t,x)}x^-+x^0+e^{\theta(t,x)}x^++C(t,x),\qquad (t,x)\in[0,1]\times\bar B\,,$$
		where $x=x^-+x^0+x^+\in X^-\oplus X^0\oplus X^+=X$. The class $\mathcal D^*$ is {\em Szulkin-admissible}, meaning that for any $\eta_1,\eta_2\in\mathcal D^*$, their juxtaposition		
 $$\eta:=\eta_1\star\eta_2:[0,1]\times\bar B\to\bar B,\quad (t,x)\mapsto\begin{cases}
			\eta_1(2t,x)&\text{if }0\leq t\leq1/2,\\
			\eta_2(2t-1,\eta_1(1,x))&\text{if }1/2\leq t\leq 1,	
		\end{cases}$$}
		belongs again to $\mathcal D^*$.  
		\end{remark}
		
	A first motivation to have considered this class is made clear by Lemma \ref{deflem} below.	
	\begin{lemma}[Deformation lemma]\label{deflem}{\em For any given $c>c_0$ and any open neighborhood $U\subset B$ of the set $K_c$ of critical points of $\Phi$ at level $c$, there exists some $0<\delta<c-c_0$ and a deformation $\eta\in\mathcal D^*$ such that
		\begin{equation}\label{etdef}
\eta(t,x)=x\text{ if }x\in\Phi_{c_0},\qquad	\eta(\{1\}\times\Phi_{c_0+\delta})=\Phi_{c_0},\qquad\eta(\{1\}\times(\Phi_{c+\delta}\backslash U))\subset\Phi_{c-\delta}\,.
		\end{equation}}
	\begin{proof}
We start by observing that $\Phi$ satisfies the Palais-Smale condition on $B$, i.e., every sequence $\{x_n\}_n\subset B$ with $\nabla\Phi(x_n)\to 0$ has subsequence converging in $B$. This is a consequence of $\nabla\Psi$ being compact, $L$ being Fredholm and $B$ being bounded.  In particular, $K_c$ is compact and one can find some open set $U_0\subset B$ with $K_c\subset U_0$ and $\dist(U_0,X\backslash U)>0$. For each $\delta>0$ we consider the set
$$R_\delta:=\Phi_{c_0+\delta}\cup\{x\in\bar B\backslash U_0:|\Phi(x)-c|\leq\delta\}\,.$$
The Palais-Smale condition implies that there are no critical points on $R_\delta$, and indeed 
$$\inf_{x\in R_\delta}\|\nabla\Phi(x)\|>0\,,$$
 provided that $\delta>0$ is small enough. Adapting a classical argument (see \cite[Lemma 3.2, p. 547]{ben}) one can find a locally Lipschitz-continuous function (sometimes called {\em pseudogradient}) $Z:\Omega_0\to X$ such that $Z(\Omega_0)$ is relatively compact in $X$ and $\sup_{x\in\Omega_0}\|Z(x)-\nabla\Psi(x)\|$ is arbitrarily small. Here $\Omega_0\subset X$ is some open set with $B\subset\Omega_0\subset\Omega$ and $\dist(\Omega_0,X\backslash\Omega)>0$. 
 After fixing $\delta$ we can select $Z$ satisfying
 \begin{equation}\label{pseudo}
\|Z(x)-\nabla\Psi(x)\|<\sigma\ \forall x\in\partial_K^+ B\,,\qquad \inf_{x\in R_\delta}\big\langle\nabla\Phi(x),Lx+Z(x)\big\rangle>0\,,
 	\end{equation}
 where $\sigma>0$ is the constant appearing in  Definition \ref{fstn}{\em (i)}.

\medbreak

Let $x(t,x_0)$ be the maximal solution of the initial value problem
$$x'=-(Lx+Z(x))\,,\qquad x(0)=x_0\,.$$
  When $x_0\in(\partial^+ B)\backslash(\partial^-B)$ this solution enters into $B$. This is a consequence of the first inequality of \eqref{pseudo}. Moreover, for any $x_0\in\bar B$ the solution will not exit $\bar B$ unless it previously enters into $\Phi_{c_0}$. Given $x_0\in B$ it may happen that $x(t,x_0)$ remains in $\bar B\backslash\Phi_{c_0}$ for each $t\in[0,1]$. Then we define $\tau(x_0):=1$. Otherwise $\tau(x_0)$ will be the first instant $t\in[0,1]$ such that $x(t,x_0)\in\Phi_{c_0}$. The function $\tau:\bar B\to\mathbb R$ is continuous and allows us to define the deformation
 $$\eta:[0,1]\times\bar B\to\bar B\,,\qquad\eta(t,x_0):=x(\tau(x_0)t,x_0)\,.$$
 
  The first assertion of \eqref{etdef} is immediate. On the other hand, the second inequality of \eqref{pseudo} can be used to show that, after possibly replacing $\delta$ by some smaller number the remaining statements of \eqref{etdef} also hold
  . The proof will be complete after showing that $\eta$ belongs to the class $\mathcal D^*$. This is a consequence of the variations of constants formula applied to the linear equation $y'=-Ly+b(t)$ with $b(t)=-Z(x(t,x_0))$. 
	\end{proof}
	\end{lemma}		
		\medbreak

		\begin{definition}[{$*$-category}]\label{catx}{The closed set $A\subset\bar B$ with $A\supset\Phi_{c_0}$ will be said to be of $*$-category $k\geq 0$, denoted $\cat^*(A)=k$, provided that $k$ is the smallest integer for which there are closed sets
$A_0,A_1,\ldots,A_k\subset A$ with $A_0\supset\Phi_{c_0}$, such that:
\begin{enumerate}
	\item[(a)] $A=A_0\cup A_1\cup\ldots A_k$,
	\item[(b)]  $A_i$ is contractible in $\bar B$ for every $i\in\{1,\ldots,k\}$,
	\item[(c)] there exists some deformation $\eta\in\mathcal D^*$ with
	\begin{equation*}
\eta(t,x)=x\text{ if }x\in\Phi_{c_0},\qquad \eta(\{1\}\times A_0)\subset\Phi_{c_0}.
	\end{equation*}
	(If such an integer $k$ does not exist, then $\cat^*(A):=+\infty$.) 
\end{enumerate}}
\end{definition}
Observe that while the deformation $\eta$ in {\em (c)} is required to belong to the class $\mathcal D^*$, the deformations -whose existence is implicit in {\em (b)}- taking the $A_i$'s for $1\leq i\leq k$ into singletons are just assumed  continuous. In this way, $A_1,\ A_2,\ldots A_k$ could be, for instance, closed balls. This fact will be important in the proof of Lemma \ref{proplsx} below.

\medbreak

The notion of $*$-category has many interesting properties, see \cite[Proposition 2.8]{szu}. For future reference we recall two of them below. In the first statement $\cat_{\bar B}$ stands for the usual Lusternik-Schnirelmann category.
\begin{lemma}[Szulkin, \cite{szu}]\label{lem1-1x}{The following hold:
	\begin{enumerate}
		\item[(i)] Let the subset $A\subset\bar B$ with $\Phi_{c_0}\subset A$ be closed, and let $U\subset B\backslash\Phi_{c_0}$ be open. If $\cat_{\bar B}(\bar U)<+\infty$ then $\cat^*(A\backslash U)\geq\cat^*(A)-\cat_{\bar B}(\bar U)$. 
		\item[(ii)] Let the subsets $A_1,A_2\subset\bar B$ with $\Phi_{c_0}\subset A_1\cap A_2$ be closed. If there exists a deformation $\eta\in\mathcal D^*$ with $\eta(t,x)=x$ for every $(t,x)\in[0,1]\times\Phi_{c_0}$ and $\eta(\{1\}\times A_1)\subset A_2$, then  $\cat^*(A_1)\leq\cat^*(A_2)$.
	\end{enumerate}}
	\end{lemma}
A third, even more important feature of the notion of $*$-category is the following:
 \begin{lemma}\label{proplsx}{It $B\subset\Omega$ is a fibered saddle neighborhood of $M$ relative to $\Phi$, then $\Phi$ has at least $\cat^*(\bar B)$ critical points in $B$.}
\begin{proof}We follow classical arguments which can be traced back to Lusternik-Schnirelmann's original memoir \cite{lussch}. For each $1\leq j\leq k:=\cat^*(B)$ we set
	$$\Gamma_j:=\{A\subset B\text{ closed such that } \Phi_{c_0}\subset A\text{ and }\cat^*(A)\geq j\}\,,$$
	and we observe that $\Gamma_j\not=\emptyset$ (because $B\in\Gamma_j$). Moreover, the combination of Lemma \ref{deflem} and Lemma \ref{lem1-1x}{\em (ii)} implies that
	$$c_j:=\inf_{A\in\Gamma_j}\left(\sup_A\Phi\right)> c_0\,,\qquad j=1,\ldots k\,.$$
	
	\medbreak
	
	Since $\Gamma_j\supset\Gamma_{j+1}$ we see that $c_1\leq c_2\leq\ldots\leq c_k$. It remains to check that if $c:=c_j=c_{j+1}=\ldots=c_{j+p}$ then $K_c$ consists at least of $p+1$ points. Assuming by contradiction that $K_c=\{x_1,\ldots,x_p\}$ (not necessarily different), let $U:=B_r(x_1)\cup\ldots\cup B_r(x_p)$ be the union of open balls centered at the points of $K_c$, small enough so that $U\subset B\backslash\Phi_{c_0}$. Pick $\epsilon>0$ and a deformation $\eta\in\mathcal D^*$ as given by Lemma \ref{deflem}, and choose some set $A\in\Gamma_{j+p}$ with $\sup_{A}\Phi<c+\epsilon$. We set 
	$$A_1:=A\backslash U,\qquad A_2:=\overline{\eta(\{1\}\times A_1)}$$
and observe that, by Lemma \ref{lem1-1x},  $$\cat^*(A_2)\geq\cat^*(A_1)\geq\cat^*(A)-\cat_{\bar B}(\bar U)\geq(j+p)-p=j\,.$$ With other words, $A_2\in\Gamma_j$, implying that $\sup_{A_2}\Phi\geq c_j=c$ and contradicting \eqref{etdef}. It completes the proof of Lemma \ref{proplsx}.
\end{proof}

\end{lemma}

The reader may inquire about the fact that the condition $c_0<\inf_{B^0}\Phi$ in Definition \ref{fstn} has not yet been used. Its role will be clear in the proof of Lemma \ref{papx}, where it will allow us to estimate $\cat^*(\bar B)$ by constructing a retraction from the sublevel set $\Phi_{c_0}$ into $\partial^-B$ (over a suitable finite-dimensional reduction of the vector bundles).

\section{From infinite to finite-dimensional vector bundles}\label{sec7.1}
In view of Lemma \ref{proplsx}, all what remains to complete the proof of Theorem \ref{th1} is to show that if $B=B_K(r^-,r^+)\subset\Omega$  is a $K$-tubular neighborhood of the compact manifold $M$ which is furthermore of fibered saddle type relative to $\Phi$, then $\cat^*(\bar B)\geq\cl(M)+1$. We recall that the map $K:M\to\mathcal K(X)$ is assumed to belong to the class $[\mathcal C_d^1]$ and to satisfy \eqref{cK}-\eqref{eq43}.

\medbreak

Going back to Proposition \ref{lem212} we see that the sets $\mathfrak X_K^\pm$ (defined as in \eqref{eu14}) are continuously-differentiable vector subbundles of $M\times X$. The result below asserts that these splittings can be conveniently approximated by finite-dimensional counterparts.

\begin{prop}\label{corro}{Under the above there are sequences
		$$\mathfrak Y_n^-\subset\mathfrak X_K^-,\ \ \mathfrak Y_n^+\subset\mathfrak X_K^+\,,\qquad n\in\mathbb N\,,$$
		of finite-dimensional, topological vector subbundles of $\mathfrak X_K^\pm$, with associated fibers $Y_{n,m}^\pm$, $m\in M$, such that the subspaces $Y_{n,m}:=Y_{n,m}^-\oplus(X_{K_m}^0)\oplus Y_{n,m}^+$ satisfy:
		\begin{enumerate}
			\item[(i)] For each $m,m_1\in M$ and $s\in\mathbb R$ one has
			 $$\lim_{n\to+\infty}d(e^{sL}[Y_{n,m}],Y_{n,m_1})=0\,,$$
			and the convergence is uniform with respect to $s$ in compact intervals and $m,m_1\in M$,
			\item[(ii)] the subspaces $Y_{n,m}$ are close to dense in the sense that, for each $x\in X$,
			$$\lim_{n\to+\infty}\dist(x,Y_{n,m})=0\,,$$  and the convergence is uniform  with respect to $x$ in compact subsets of $X$ and $m\in M$.
			
		\end{enumerate}
	}\end{prop}
Proposition \ref{corro} is the fifth -and last- auxiliary result whose proof is postponed to the Appendix (Section \ref{App}), at the end of the paper.

\section{From infinite to finite-dimensional category and cuplength}\label{secf} Let us begin by recalling the finite-dimensional notion of relative category with respect to a subset. Given any metric space $E$ and closed subsets $A,E_0\subset E$ with $A\supset E_0$ it is said that {\em $A$ is of category $k\geq 0$ in $E$ relative to $E_0$}, denoted $\cat_{E,E_0}(A)=k$, provided that $k$ is the smallest integer for which there are closed sets
$\hat A_0,\hat A_1,\ldots,\hat A_k\subset E$ with $\hat A_0\supset E_0$, such that:
\begin{enumerate}
	\item[{\em (\,$\hat a$)}] $A=\hat A_0\cup\hat A_1\cup\ldots\hat A_k$,
	\item[{\em (\,$\hat b$)}]  $\hat A_i$ is contractible in $E$ for every $i\in\{1,\ldots,k\}$,
	\item[{\em (\,$\hat c$)}] there exists some deformation $\hat\eta$ of $E$ with
	\begin{equation}\label{k1}
\hat\eta(t,x)=x\text{ if }x\in E_0,\qquad \hat\eta(\{1\}\times\hat A_0)\subset E_0.
	\end{equation}
\end{enumerate}
	If such an integer $k$ does not exist, $\cat_{E,E_0}(A):=+\infty$. In case $E_0=\emptyset$ then necessarily $\hat A_0=\emptyset$ and thus $\cat_{E,\emptyset}\equiv\cat_E$ becomes the usual Lusternik-Schnirelmann category. An important difference between the $*$-category introduced in Definition \ref{catx} and this new notion of category lies in the fact that here $\hat\eta$ is {\em any} deformation of $E$ while in Definition \ref{catx} the deformation $\eta$ was assumed to belong to the class $\mathcal D^*$. We shall show the following:
\begin{lemma}\label{papx}{Under the framework of Proposition \ref{corro}, for $n$ big enough the finite-dimensional, topological vector subbundle ${\mathfrak Y^-}:=\mathfrak Y_n^-$ satisfies that 		$\cat_{E,\,E_0}(E)\leq\cat^*(\bar B)$. Here,
		\begin{eqnarray}
			E:=\{m+y^-:(m,x^-)\in{\mathfrak Y^-},\ \|y^-\|\leq r^-\},\label{tlb1}\\ E_0:=\{m+y^-:(m,y^-)\in\mathfrak Y^-,\ \|y^-\|=r^-\}\,.\label{tlb2}
		\end{eqnarray}}
	\end{lemma}
As a preparation for the proof we observe that there exists a positive number $0<\rho<r^-$ such that
		\begin{equation}\label{inc654}
		\Phi_{c_0}\subset\mathcal E_\rho:=\{m+v^-+v^+\in\bar B:\|v^-\|\geq\rho\}\,. 	
		\end{equation}
(Keeping our previous notation the points of $x\in\bar B$ are written in the form   $x=m+v^-+v^+$ where $m\in M$ and $v^\pm\in X_{K_m}^\pm$ satisfy $\|v^\pm\|\leq r^\pm$).	In order to argue this assertion we point out that $\Phi$ is Lipschitz-continuous on $\bar B$; this can be checked using that $\nabla\Phi$ is bounded on $B$, and, while $B$ may not be convex, it is fiberwise convex and $M$ is a compact manifold. Together with the inequality $\inf_{B^0}\Phi>c_0$ imposed in Definition \ref{fstn}{\em (ii)} one can deduce the existence of $\rho$ as required. 
		
		\medbreak

		Also, it will be convenient to introduce some notation. The projections of $\bar B$ onto the basis and the fibers of the bundle will be denoted by
		$$p_M,p_\pm:\bar B\to X\,,\qquad p_M(m+v^-+v^+)=m\,,\qquad p_\pm(m+v^-+v^+)=v^\pm\,,$$

		\medbreak
		
		Since  $B=B_K(r^-,r^+)$ is a  $K$-tubular neighborhood of $M$, the discussion in Section \ref{sec5} implies that the three projections are continuous. Given a deformation $\eta$ of $\bar B$ we define
		$$\eta_M=p_M\circ\eta,\qquad\eta_\pm=p_\pm\circ\eta\,.$$
		After these preliminaries we are ready to present the proof.
		
			\begin{proof}[Proof of Lemma \ref{papx}] Pick some {\em big} natural $n\in\mathbb N$ (the meaning of the word `big' will be made precise later), and let $\mathfrak Y^-$, $E$ and $E_0$ be defined as above.  If $\cat^*(\bar B)=+\infty$ then there is nothing to prove. Therefore  we assume that $k:=\cat^*(\bar B)<+\infty$, and pick closed subsets
		$A_0,A_1,\ldots,A_k\subset\bar B$ as in Definition \ref{catx} for $A:=\bar B$. We define 
		$$\hat A_i:=A_i\cap E\,,\qquad 0\leq i\leq k\,.$$
		Notice that $E_0\subset\hat A_0$. This follows from the inclusions $$E_0=\partial_K^-B\cap E\subset\partial_K^-B\subset\Phi_{c_0}\,,$$ the last one being again a consequence of Definition \ref{fstn}{\em (ii)}.
		
		\medbreak
		
		To prove the lemma it is sufficient to show that the sets $\hat A_i$ satisfy the conditions {\em (\,$\hat a$)}, {\em (\,$\hat b$)}, {\em (\,$\hat c$)} with $A=E$. Since {\em (\,$\hat a$)} is obvious we start with {\em (\,$\hat b$)}. Given $i\geq 1$ we consider a deformation $\eta$ of $A_i$ in $\bar B$ such that $\eta(\{1\}\times A_i)$ is a singleton. We define, for every  $t\in[0,1]$ and $m+y^-\in\hat A_i$,
		\begin{equation*}
		\hat\eta(t,m+y^-):=m_1+\Pi_{Y_{n,m_1}^-}\big[\eta_-(t,m+y^-)\big]\,,
		\end{equation*}
		with $m_1=\eta_M(t,m+y^-)$. Now $\hat\eta$ is a deformation of $\hat A_i$ on $E$ taking $\{1\}\times\hat A_i$ into a singleton, and so $\hat A_i$ is contractible in $E$.

		\medbreak
		
The proof of  (\,$\hat c$) is more delicate and needs the assumption that $n$ is big, which has not been used yet. Let $\eta\in\mathcal D^*$ be in the conditions of Definition \ref{catx}{\em (c)} with $A=\bar B$, and let the deformation $\bar\eta$ of $E$ be defined by
\begin{equation*}
	\bar\eta(t,m+y^-):=m_1+\Pi_{Y_{n,m_1}^-}\big[\eta_-(t,m+y^-)\big]\,,
\end{equation*}
 with $m_1=\eta_M(t,m+y^-)$. We shall prove that for $n$ big enough one has the inclusion
\begin{equation}\label{claimx012}
	\bar\eta(\{1\}\times\hat A_0)\subset\mathcal E_{\rho/2}\cap E=:\widehat{\mathcal E}_{\rho/2}\,.
	\end{equation}
Once it has been shown, it suffices to juxtapose $\bar\eta$ with some deformation $h$ of $E$ such that 
$$h(t,x)=x\text{ if }x\in E_0\,,\qquad h(\{1\}\times\widehat{\mathcal E}_{\rho/2})=E_0\,,$$
for it is clear that $\hat\eta:=\bar\eta\star h$ is a deformation of $E$ satisfying {\em (\,$\hat c$)}.

\medbreak

In order to check \eqref{claimx012} we recall that  $\bar\eta([0,1]\times E)\subset E$ and so we only have to show the inequality 
\begin{equation}\label{xy12}
\left\|\Pi_{Y_{n,m_1}^-}\big[\eta_-(1,m+y^-)\big]\right\|\geq\frac{\rho}{2}\,,\qquad m+y^-\in E\,,
\end{equation}
but by combining the inclusion $\eta(\{1\}\times A_0)\subset\Phi_{c_0}$ with \eqref{inc654} we see that
\begin{equation}\label{xy13}
\left\|\eta_-(1,m+y^-)\right\|\geq\rho\,,\qquad m+y^-\in E\,.
\end{equation}

Pick functions $\theta=\theta(t,x)$ and $C=C(t,x)$ as in \eqref{eq9}-\eqref{eq07} and observe that
\begin{equation}\label{etamen}
\eta_-(1,m+y^-)=\tilde C(m+y^-)+\Pi_{X_{K_{m_1}}^-}\left[\exp(\theta(1,m+y^-)L)y^-\right]\,,
\end{equation}
the function $\tilde C:\bar B\to X$ being given by $$\tilde C(x)=\Pi_{X_{K_{m_1}}^-}\Big[\exp\big(\theta(1,x)L\big)m+C(1,x)-m_1\Big]\,,$$
(we set $m_1:=\eta_M(1,x)$). Consequently, \eqref{etamen} gives the equality 
\begin{equation}\label{eK:1}
(\id_X-\Pi_{Y_{n,m_1}^-})\eta_-(1,m+y^-)=s_1+s_2\,,
\end{equation}
where 
\begin{equation}\label{eK:3}
	s_1=(\id_X-\Pi_{Y_{n,m_1}^-})\tilde C(m+y^-)=(\id_X-\Pi_{Y_{n,m_1}})\tilde C(m+y^-)\,,	
\end{equation}
and also
\begin{multline}\label{eK:2}
s_2=(\id_X-\Pi_{Y_{n,m_1}^-})\Pi_{X_{K_{m_1}}^-}\big[\exp(\theta(1,m+y^-)L)y^-\big]=\\=\Pi_{X_{K_{m_1}}^-}(\Pi_{(\exp(\theta(1,m+y^-)L)[Y_{n,m}]}-\Pi_{Y_{n,m_1}})\big[\exp(\theta(1,m+y^-)L)y^-\big]\,,
\end{multline}
 where we have used that $\Pi_{Y_{n,m_1}^-}\circ\Pi_{X_{K_{m_1}}^-}=\Pi_{Y_{n,m_1}^-}=\Pi_{X_{K_{m_1}}^-}\circ\Pi_{Y_{n,m_1}}$.

\medbreak

Notice now that $\tilde C$ is compact on $\bar B$, i.e., the closure $K$ of $\tilde C(\bar B )$ is compact in $X$. Thus, \eqref{eK:3} gives
$$\|s_1\|\leq\max_{x\in K,\,m\in M}\dist(x,Y_{n,m})\to 0\text{ as }n\to+\infty\,,$$
where we have used Proposition \ref{corro}{\em (ii)}. On the other hand, it follows from \eqref{eK:2} that
\begin{multline*}
\|s_2\|\leq\left\|(\Pi_{(\exp(\theta(1,m+y^-)L)[Y_{n,m}]}-\Pi_{Y_{n,m_1}})\big[\exp(\theta(1,m+y^-)L)y^-\big]\right\|\leq\\ \leq\gamma\, d\big((\exp(\theta(1,m+y^-)L)[Y_{n,m}],Y_{n,m_1}\big)\,,
\end{multline*}
with $\gamma=r^-\exp\big(\sup_{[0,1]\times\bar B}|\theta|\big)$. Here we have used that $\|L\|_{\mathscr L(X)}=1$, by \eqref{eu12}. Consequently, Proposition \ref{corro}{\em (i)} implies that $\|s_2\|$ will also be small for big $n$, uniformly with respect to $t\in[0,1]$ and independently of $m+y^-\in E$.

\medbreak

Now, going back to \eqref{eK:1} we see that if $n$ is big enough one has
$$\|(\id_X-\Pi_{Y_{n,m_1}^-})\eta_-(1,m+y^-)\|<\frac{\rho}{2}\,,\qquad m+y^-\in E\,,$$
and, together with \eqref{xy13} we arrive to \eqref{xy12}. It completes the proof.
	\end{proof}

We recall that the {\em cuplength } $\cl(M)$ of $M$ (associated to singular cohomology with $\mathbb Z_2$ coefficients) is the maximal integer $k$ for which there are natural numbers $q_1,\ldots,q_k\geq 1$ and cohomology classes $\alpha_j\in H^{q_j}(M)$, $1\leq j\leq k$, such that $\alpha_1\cup\ldots\alpha_k\not=0$. By Poincar\'e's duality theorem, $H^n(M)\cong H_0(M)\cong\mathbb Z_2$, and so $\cl(M)\geq 1$.

\begin{proof}[Proof of Theorem \ref{th1}]It suffices now to combine lemmas \ref{proplsx}-\ref{papx} with the inequalities
	$$\cat_{E,E_0}(E)\geq\cl(E,E_0)\geq 1+\cl(M)\,,$$
where $E,E_0$ are the sets in \eqref{tlb1}-\eqref{tlb2}. Now, the first inequality is a well-known relation between relative category and cuplength, see e.g.  \cite[Proposition 2.6, p. 728]{szu} or \cite[Theorem 1, p. 97]{fouwil2}. The second part is also well-known and follows from Thom's isomorphism theorem relating the cohomology groups of $H^q(M)$ with the relative cohomology groups $H^{q+d}(E,E_0)$ of the spheric bundle $E$ with respect to its boundary $E_0$. See, e.g., \cite[Lemma 3.7, p. 733]{szu}. 
\end{proof}

\section{Appendix}\label{App}

We prepare the proof of Proposition \ref{lem4.1} by reviewing some well-known results on the   spectral theory of selfadjoint operators in the Hilbert space $X$, for which we refer to \cite[\S 6.7-6.8]{fri}, \cite[\S 132-134]{riebel}. We recall that the family $\{E_\lambda\}_{\lambda\in\mathbb R}$ of continuous orthogonal projections in $X$ is said to be a {\em spectral family} provided that it is nondecreasing, right-continuous and bounded, in the sense that:
\begin{enumerate}
	\item[{\em (a).}] $E_{\lambda_
		2}\circ E_{\lambda_1}=E_{\lambda_
		1}\circ E_{\lambda_2}=E_{\lambda_1}$ if $\lambda_1<\lambda_2$ {\em (nondecreasing)},
	\item[{\em (b).}] 
	$\lim_{\lambda_2\to(\lambda_1)_+}E_{\lambda_2}=E_{\lambda_1}$ in $\mathscr L(X)$, for every $\lambda_1\in\mathbb R$ {\em (right-continuous)},
	\item[{\em (c).}] There exists some constant $\Lambda>0$ such that
	$E_\lambda=0$ if $\lambda\leq-\Lambda$ and $E_\lambda=\id_X$ if $\lambda\geq\Lambda$ {\em (bounded)}.
	\end{enumerate}
	
	For such an spectral family and a given continuous function $f:\mathbb R\to\mathbb R$, the integral $\int_{-\infty}^{+\infty}f(\lambda)dE_\lambda$ is the operator in $\mathscr L(X)$ obtained by a limiting process from the sums of Riemann-Stieltjes type
	\begin{equation}\label{Tx1}
	\sum f(\xi_k)(E_{\lambda_{k+1}}-E_{\lambda_k})\,,
	\end{equation}
	where $\{\lambda_k\}$ is a partition of $[-\Lambda,\Lambda]$ and each $\xi_k$ lies in the interval $[\lambda_k,\lambda_{k+1}]$.
	
	\medbreak
	
	The {\em fundamental theorem on spectral decomposition} (see, e.g., \cite[p. 320]{riebel}) asserts that {\em for every continuous selfadjoint operator $T:X\to X$ there exists a unique spectral family $\{E_\lambda\}_{\lambda\in\mathbb R}$  such that $T=\int_{\mathbb R}\lambda\, dE_\lambda$\,.}	In addition, for each $x\in X$ one has the equality
		\begin{equation}\label{txx123}
	\langle Tx,x\rangle=\int_{-\infty}^\infty\lambda\, d\xi_x(\lambda)\,,
	\end{equation}
	where $\xi_x(\lambda)=\langle E_\lambda x,x\rangle$ and the integral is to be considered in the Riemann-Stieltjes sense. Notice that the functions $\xi_x:\mathbb R\to\mathbb R$ are nondecreasing, continuous from the right, and satisfy 
$$\xi_x(\lambda)=0\text{ if }\lambda\leq-\Lambda\,,\qquad \xi_x(\lambda)=\|x\|^2 \text{ if }\lambda\geq\Lambda\,,$$
	so that the right hand side of \eqref{txx123} is well-defined.
	
	\medbreak
	
We observe that as a consequence of {\em (a)} each $E_\lambda$ commutes with the sums \eqref{Tx1}, implying that
	\begin{equation}\label{commute}
		T\circ E_\lambda=E_\lambda\circ T\,,\qquad\lambda\in\mathbb R\,.
	\end{equation}

\medbreak

	On the other hand, if $x\in E_{\lambda_0}(X)$ then $\xi_x(\lambda)=\|x\|^2$ for every $\lambda\geq\lambda_0$, implying that
\begin{equation}\label{ineq345}
\langle Tx,x\rangle=\int_{-\infty}^{\lambda_0}\lambda\, d\xi_x(\lambda)\leq\lambda_0\int_{-\infty}^{\lambda_0}d\xi_x(\lambda)=\lambda_0\|x\|^2\,,
\end{equation}
where we have used that $\int_{-\infty}^{+\infty} d\xi_x(\lambda)=\|x\|^2$. Similarly one checks the inequality
\begin{equation}\label{ineq456}
	\langle Ty,y\rangle>\lambda_0\|y\|^2\,,\qquad 0\not=y\in\ker(E_{\lambda_0})=E_{\lambda_0}(X)^\bot\,.
\end{equation}

	\medbreak
	
In addition we point out that the spectrum of $T$ can be characterized in terms of its associated spectral family $\{E_\lambda\}$. Indeed, a number $\lambda_0\in\mathbb R$ belongs to the resolvent set of $T$ if and only if the map $\lambda\in\mathbb R\mapsto E_\lambda\in\mathscr L(X)$ is locally constant around $\lambda_0$. On the other hand the eigenvalues $\lambda_0$ of $T$ are characterized as the points of discontinuity of this map. In this case $\ker(T-\lambda_0\id_X)$ is the image of the orthogonal projection $E_{\lambda_0}-E_{\lambda_0^-}$, with $E_{\lambda_0^-}=\lim_{\lambda\nearrow\lambda_0}E_\lambda$.
	
	\medbreak

	\begin{proof}[Proof of Proposition \ref{lem4.1}]We claim that if $0\in\sigma(L)$ then it is an isolated eigenvalue. More precisely, 
		\begin{equation}\label{eqx1y}
		L(X)=(\ker L)^\bot
		\end{equation}
		 and there exists some $r>0$ such that 
		\begin{equation}\label{eqL7}
	\big(\sigma(L)\backslash\{0\}\big)\cap[-r,r]=\emptyset\,.
		\end{equation}
		To prove these statements we first observe that $\tilde X=L(X)$ is closed in $X$, because $L$ is a Fredholm operator. Then, since it is also selfadjoint, $\tilde X=(\ker L)^\bot$ and the orthogonal splitting $X=\tilde X\oplus\ker L$ is invariant under $L+\lambda\id_X$ for each $\lambda\in\mathbb R$. For $\lambda=0$ the restriction of $L$ to $\tilde X$ defines a continuous and bijective operator $L_{\tilde X}:\tilde X\to\tilde X$. Hence, $L_{\tilde X}$ is a topological isomorphism. This property is open and therefore $L_{\tilde X}+\lambda\id_{\tilde X}:\tilde X\to\tilde X$ will also be a topological isomorphism when $|\lambda|$ is small, say $|\lambda|\leq r$. For $0<|\lambda|\leq r$ we can express $L+\lambda\id_X:X\to X$ as a direct sum of topological isomorphisms. In consequence, $[-r,r]\backslash\{0\}$ is contained in the resolvent set, thus showing \eqref{eqL7}. 
		
		\medbreak
		
		 Let $\{E_\lambda\}$ be the spectral family associated to $L$. Then it follows from  \eqref{eqL7} that the map $\lambda\mapsto E_\lambda$ is constant, both on $[-r,0[$ and on $[0,r]$. With other words,
		 \begin{equation}\label{todo3}
		 E_\lambda=E_{-r}\text{ if }-r\leq\lambda<0\,,\qquad\qquad E_\lambda=E_{r}\text{ if }0\leq\lambda\leq r\,. 
		 \end{equation}
				
		It is now clear that $\ker L=(E_r-E_{-r})(X):=X^0$. Set 
		\begin{equation}\label{eqxg}
		X^-:=E_{-r}(X)\,,\qquad X^+:=\ker E_r\,.	
		\end{equation}
	It follows from  \eqref{commute} that $X^\pm$ are $L$-positively invariant, and \eqref{eqx1y} implies that they are indeed $L$-invariant. Moreover, the inequalities \eqref{inex} follow from \eqref{ineq345}-\eqref{ineq456} with $\lambda_0=\pm r$.

\medbreak

It remains to show the uniqueness part of the statement. To this aim assume that $X=\widetilde X^-\oplus(\ker L)\oplus\widetilde X^+$ is another orthogonal splitting under the conditions of Proposition \ref{lem4.1}. Then $L$ induces two continuous, selfadjoint operators $L^\pm:\widetilde X^\pm\to\widetilde X^\pm$. Each of these operators has an associated spectral family $\{\widetilde E_\lambda^{\pm}\}_{\lambda\in\mathbb R}$. For sufficiently small $r_1>0$ the orthogonal projections $\widetilde E_\lambda^{\pm}:\widetilde X^\pm\to \widetilde X^\pm$ satisfy
\begin{equation}\label{exqx}
\widetilde E_\lambda^{+}=0\text{ if }\lambda\leq r_1\,,\qquad\widetilde  E_\lambda^{-}=\id_{\widetilde X^-}\text{ if }\lambda\geq-r_1\,.
\end{equation}

We consider the family of projections $\{\widetilde E_\lambda\}\subset\mathscr L(X)$  by setting: 
$$\widetilde E_\lambda:=\begin{cases}
\widetilde E_\lambda^-\circ\Pi_{\widetilde X^-}&\text{ if }\lambda<0\,,\\
\Pi_{\widetilde X^-}+\Pi_{\ker L}+\widetilde E_\lambda^{+}\circ\Pi_{\widetilde X^+}&\text{ if }\lambda\geq 0\,.
\end{cases}$$

One easily checks that this is a spectral family associated to $L$, and by uniqueness $\widetilde E_\lambda=E_\lambda$ for every $\lambda\in\mathbb R$. Consequently, for small $r_1>0$ the combination of \eqref{todo3}, \eqref{eqxg} and \eqref{exqx} gives  $$X^-=E_{-r_1}(X)=\widetilde X^-,\qquad X^+=\ker E_{r_1}=\widetilde X^+\,,$$ thus concluding the argument.
\end{proof}

We continue now with our discussion previous to the proof of Proposition \ref{lem4.1}. An important property of the spectral family associated to the continuous selfadjoint operator $T\in\mathscr L(X)$ is that for every real numbers $\lambda_1<\lambda_2$ with $\lambda_j\not\in\sigma(T)$ one has
\begin{equation}\label{proj}
	E_{\lambda_2}-E_{\lambda_1}=\frac{1}{2\pi i}\int_\gamma(z\id_X-T)^{-1}dz\,.
\end{equation}
The right-hand side is a Cauchy-type, operator-valued integral. Equality \eqref{proj} holds for any continuously-differentiable, positively-oriented Jordan curve $\gamma\subset\mathbb C\backslash\sigma(T)$ satisfying
$$\intt(\gamma)\cap\sigma(T)=[\lambda_1,\lambda_2]\cap\sigma(T)\,,$$ 
where $\intt(\gamma)$ stands for the (open) bounded connected component of $\mathbb C\backslash\gamma$. See, e.g., \cite[\S 148]{riebel}.

\medbreak

In particular, under the framework of Proposition \ref{lem4.1} we see that \eqref{eqL7} implies the existence of continuously-differentiable, positively-oriented Jordan curves $\gamma^0,\gamma^-,\gamma^+\subset\mathbb C\backslash\sigma(L)$ satisfying
$$\intt(\gamma^0)\cap\sigma(L)\subset\{0\}\,,\qquad\intt(\gamma^\pm)\cap\sigma(L)=\sigma(L)\cap\mathbb R^\pm\,,$$
where $\mathbb R^\pm$ denote the open intervals of positive and negative numbers. In this situation \eqref{proj} gives
\begin{equation}\label{ker123}
	\Pi_{X^0}=\frac{1}{2\pi i}\int_{\gamma^0}(z\id_X-L)^{-1}dz\,,\qquad \Pi_{X^\pm}=\frac{1}{2\pi i}\int_{\gamma^\pm}(z\id_X-L)^{-1}dz\,.
\end{equation}

\begin{proof}[Proof of Proposition \ref{lem501}] It will be divided into three steps. Several times we will extract subsequences that will be denoted as the original sequence. Choose  convergent sequences $K_n\to K_*$ in $\mathcal K_p(X)$, we must prove that $X_{K_n}^\pm\to X_{K_*}^\pm$ and $X_{K_n}^0\to X_{K_*}^0$. This last assertion will be checked first. 
		
		\medbreak
		
		{\em First step: $X_{K_n}^0\to X_{K_*}^0$}. For each $n\in\mathbb N$ we pick some orthonormal basis $\{b_n^{(1)},\ldots,b_n^{(p)}\}$ of $X_{K_n}^0$. Since $(L+K_*)b_n^{(j)}\to 0$, the linear operator $L$ is Fredholm and $K_*$ is compact, after passing to a subsequence we may assume that $b_n^{(j)}\to b_*^{(j)}$, $j=1,\ldots,p$. Then $\{b_*^{(1)},\ldots,b_*^{(p)}\}$ is an orthonormal system contained in $X_{K_*}^0$. Since $\dim X_{K_*}^0=p$ we conclude that it is indeed a basis. The orthogonal projection into $X_{K_n}^0$ can be expressed as 
		$$\Pi_{X_{K_n}^0}x=\sum_{j=1}^p\langle x,b_n^{(j)}\rangle\, b_n^{(j)}\,,\qquad x\in X\,,$$
		and so $\Pi_{X_{K_n}^0}\to\Pi_{X_{K_*}^0}$ in $\mathscr L(X)$. It proves the first step.

		\medbreak
		
		{\em Second step: There exists some $r_1>0$ such that 
			\begin{equation}\label{r1xc}
			(\sigma(L+K_n)\backslash\{0\})\cap[-r_1,r_1]=\emptyset,\qquad\text{ for every }n\in\mathbb N\,.
			\end{equation}
			}

		By contradiction assume that, after possibly passing to a subsequence, it is possible to find numbers $0\not=\lambda_n\in\sigma(L+K_n)$ with $\lambda_n\to 0$. We know that $\sigma_{ess}(L+K_n)=\sigma_{ess}(L)$ is closed in $\mathbb R$, see \cite[p. 363]{riebel}. Since $0\not\in\sigma_{ess}(L)$ we conclude that $\lambda_n$ must be an eigenvalue of $L+K_n$ for large $n$. Select an associated eigenvector $u_n\in X$ with $\|u_n\|=1$ and $(L+K_n)u_n=\lambda_n u_n$. Since $\lambda_n\not=0$ we observe that
		\begin{equation}\label{RO1}
			u_n\in (L+K_n)(X)=(X_{K_n}^0)^\bot\,,
		\end{equation}
	but we also have that
	\begin{equation}\label{RO2}
		(L+K_*)u_n\to 0\,.
	\end{equation}	
Combining again the fact that $L$ is Fredholm and $K_*$ is compact it is possible to extract a converging subsequence $u_n\to u_*$. Passing to the limit in \eqref{RO1} we deduce that $u_*\in(X_{K_n}^0)^\bot$. Here we are using that $d((X_{K_n}^0)^\bot,(X_{K_*}^0)^\bot)=d(X_{K_n}^0,X_{K_*}^0)\to 0$, by the first step. On the other hand, passing to the limit in \eqref{RO2} we see that $u_*\in X_{K_*}^0$. This is the searched contradiction because $u_*\in X_{K_*}^0\cap(X_{K_*}^0)^\bot$ and $\|u_*\|=1$.

\medbreak

	{\em Third step:  $X_{K_n}^\pm\to X_{K_*}^\pm$}. We combine the second part of \eqref{ker123} -with $L$ being replaced by $L+K_n$ and $L+K_*$- with the theorem of continuous dependence of integrals depending on parameters. The crucial observation is that the curves $\gamma^\pm$ can be chosen independently of $n$, and simultaneously valid for the limit operator $L+K_*$. This is a consequence of the second step and the inclusions 
	\begin{equation}\label{eux21}
	\sigma(L+K_n)\subset[-R,R]\ \forall n\in\mathbb N\,,\qquad\text{ with }R=\sup_n\|L+K_n\|_{\mathscr L(X)}\,.
	\end{equation}
	\end{proof}
	
	Arguing as in \eqref{r1xc}-\eqref{eux21}, the compactness of $M$ allows us to find numbers $0<r_2<R_2$ such that 
	\begin{equation}\label{eux1}
		\sigma(L+K_m)\subset]-R_2,-r_2[\,\cup\,\{0\}\,\cup\,]r_2,R_2[,\qquad\text{ for every }m\in M\,.
	\end{equation}
	
	\begin{proof}[Proof of Proposition \ref{prop3}] 
				Since $M$ is a compact manifold of class $C^1$ and $\mathcal K(X)$ is a Banach space, a regularization argument shows the existence of a sequence 
		$$\widetilde K^{(n)}\in C^1(M,\mathcal K(X)),\qquad n=1,2,\ldots$$ converging uniformly to $K$ on $M$. While there is no guarantee that the operators $\widetilde K^{(n)}_m$ lie in $\mathcal K_{d}(X)$, a compactness argument shows the existence of some index, say $n_0\in\mathbb N$, such that, if  $n\geq n_0$ then 
		$$\pm r_2\not\in\sigma\big(L+\widetilde K^{(n)}_m\big)\,,\qquad\forall m\in M\,.$$
		
		Pick some continuously-differentiable, positively-oriented Jordan curve $\gamma\subset\mathbb C$ with $\gamma\cap\mathbb R=\{-r_2,r_2\}$ and such that its interior region $R_b(\gamma)$ satisfies
		$$\intt(\gamma)\cap\mathbb R=]-r_2,r_2[\,.$$
		Applying \eqref{proj} with $T=L-\widetilde K_m^{(n)}$ we see that for every $n\geq n_0$ and $m\in M$, the linear map
		\begin{equation}\label{Kmnx}
		\frac{1}{2\pi i}\int_{\gamma}\big(z\id_X-L-\widetilde K_m^{(n)}\big)^{-1}dz=:\Pi_{V_{n,m}}	
		\end{equation}
		is the orthogonal projection on some closed subspace $V_{n,m}\subset X$, and $V_{n,m}\to\ker(L+K_m)$ as $n\to+\infty$, uniformly with respect to $m\in M$. Thus, for every $n\geq n_0$ and $m\in M$, the map $K^{(n)}_m\in\mathscr L(X)$ defined by 
		$$K_m^{(n)}x:=\begin{cases}
			-Lx
			&\text{if }x\in{V_{n,m}}\\ 
			\widetilde K_m^{(n)}x&\text{if }x\in{V_{n,m}^\bot}\\ 
		\end{cases}$$
		belongs to $\mathcal K_d(X)$. Furthermore one checks that $K_m^{(n)}\to K_m$ as $n\to+\infty$, uniformly with respect to $m\in M$. 
		
		\medbreak
		
		To complete the proof we must check that the maps $K^{(n)}:M\to\mathscr L(X)$ defined as above are continuously differentiable. To this end we observe that 
		$$K^{(n)}_m=-L\circ\Pi_{V_{n,m}}+\widetilde K_m^{(n)}\circ(\id_X-\Pi_{V_{n,m}})\,,\qquad n\geq n_0,\ m\in M\,.$$
		For each $n\geq n_0$ the map $m\mapsto\Pi_{V_{n,m}}$ is continuously-differentiable as a consequence of \eqref{Kmnx} and the theorem of continuous dependence with respect to parameters. The result follows.
\end{proof}	

	\begin{proof}[Proof of Proposition \ref{lem212}] Pick some $r_2>0$ small enough so that \eqref{eux1} holds and set $R:=\max_{m\in M}\|K_m\|+1>r_2$. In this way one has
		$$\sigma(L+K_m)\subset]-R,-r_2[\,\cup\,\{0\}\,\cup\,]r_2,R[\qquad\text{for every }m\in M\,.$$
		
It allows us to fix continuously-differentiable, positively-oriented Jordan curves $\gamma^-,\gamma^0,\gamma^+\subset\mathbb C$ with $$\gamma^0\cap\sigma(L+K_m)=\emptyset=\gamma^\pm\cap\sigma(L+K_m)\qquad\text{ for every }m\in M\,,$$ and moreover,
		$$\intt(\gamma^0)\cap\sigma(L+K_m)\subset\{0\}\,,\quad\intt(\gamma^\pm)\cap\sigma(L+K_m)=\mathbb R^\pm\cap\sigma(L+K_m)\,,\qquad \forall m\in M\,.$$	
Applying \eqref{ker123} with $L+K_m$ in the place of $L$ we see that, for every $m\in M$,
			\begin{equation}\label{eq12}
				\Pi_{X^\pm_{K_m}}=\frac{1}{2\pi i}\int_{\gamma^\pm}\big(z\id_X-L-K_m\big)^{-1}dz\,,\quad \Pi_{X^0_{K_m}}=\frac{1}{2\pi i}\int_{\gamma^0}\big(z\id_X-L-K_m\big)^{-1}dz\,. 
			\end{equation}
			By the theorem of smooth dependence of integrals with respect to parameters, the maps $\Pi_{X^\pm_{K}}$, $\Pi_{X^0_{K}}:M\to\mathscr L(X)$ defined in this way are continuously-differentiable, thus giving rise to continuously-differentiable vector subbundles $\mathfrak X_K^\pm,\mathfrak X_K^0$ of $M\times X$. It completes the proof.  
		\end{proof}	
The Grassmannian family  $\mathcal G(X)$ of closed subspaces of the Hilbert space $X$ admits a pseudodistance $\delta$ which will play a role in the proof of Proposition \ref{corro}. For any $V_1,V_2\in\mathcal G(X)$ we define
\begin{equation*}
	\delta(V_1,V_2):=\sup_{v_1\in V_1,\ \|v_1\|\leq 1}\dist(v_1,V_2).\end{equation*}
In other words, $\delta(V_1,V_2)$ is the  $\mathscr L(V_1,X)$-norm of the restriction to $V_1$ of the orthogonal projection $\Pi_{V_2^\bot}:X\to V_2^\bot$. This  nonnegative function satisfies the triangle inequality, i.e. 
\begin{equation}\label{triangin}
\delta(V_1,V_3)\leq\delta(V_1,V_2)+\delta(V_2,V_3)\,,\qquad V_1,V_2,V_3\in\mathcal G(X)\,.
\end{equation}
To prove this fact it suffices to observe that for each $v_1\in V_1$ with $\|v_1\|\leq 1$ one has
$$\dist(v_1,V_3)\leq\|v_1-v_2\|+\dist(v_2,V_3)\leq\delta(V_1,V_2)+\delta(V_2,V_3)\,,$$
where $v_2=\Pi_{V_2}v_1$. 

\medbreak

Notice also that $\delta(V_1,V_2)=0\Leftrightarrow V_1\subset V_2$.
 In addition we claim that
\begin{equation}\label{idete}
	d(V_1,V_2)\leq\delta(V_1,V_2)+\delta(V_2,V_1)\leq2d(V_1,V_2)\,,\qquad V_1,V_2\in\mathcal G(X)\,.
	\end{equation}
The right equality is direct. In order to check the left one  we pick some $x\in X$ with $\|x\|\leq 1$ and decompose it as $x=x_1+x_1^\bot$ with $x_1=\Pi_{V_1}x\in V_1$ and $x_1^\bot\in V_1^\bot$. Then, 
$$\|\Pi_{V_2}x_1^\bot\|^2=\langle x_1^\bot,\Pi_{V_2}x_1^\bot\rangle=\langle x_1^\bot,(\id_X-\Pi_{V_1})\Pi_{V_2}x_1^\bot\rangle\leq\delta(V_2,V_1)\,\|\Pi_{V_2}x_1^\bot\|\,,$$
and this is to be combined with
$$\|(\Pi_{V_1}-\Pi_{V_2})x\|=\|x_1-\Pi_{V_2}x_1-\Pi_{V_2}x_1^\bot\|\leq\|(\id_X-\Pi_{V_2})x_1\|+\|\Pi_{V_2}x_1^\bot\|\,.$$ 
A curiosity arising from \eqref{idete} is that $\hat d(V_1,V_2):=\delta(V_1,V_2)+\delta(V_2,V_1)$ defines a new distance in $\mathcal G(X)$ which is equivalent to $d$.

\medbreak

In addition we point out that for any closed subspaces $V_1,V_2\in\mathcal G(X)$ with $\delta(V_1,V_2)<1$, the linear subspace $\widetilde V_1:=\Pi_{V_2}(V_1)$ is closed and satisfies
	\begin{equation}\label{dxv3}
d(V_1,\widetilde V_1)\leq\frac{2\delta(V_1,V_2)}{\sqrt{1-\delta(V_1,V_2)^2}}\,.	
\end{equation}
To show this we apply the Pythagoras theorem to obtain that, for each $v_1\in V_1$,
\begin{equation}\label{k312}\|\Pi_{\widetilde V_1}v_1\|=\|\Pi_{V_2}v_1\|=\sqrt{\|v_1\|^2-\|(\id_X-\Pi_{V_2})v_1\|^2}\geq\sqrt{1-\delta(V_1,V_2)^2}\ \|v_1\|\,.
	\end{equation}
In consequence $\Pi_{\widetilde V_1}$ defines a topological isomorphism from $V_1$ to $\widetilde V_1$. In particular $\widetilde V_1$ is complete and hence closed in $V_2$. Moreover, for every $\tilde v_1\in\widetilde V_1$ one has
\begin{equation*}
\dist(\tilde v_1, V_1)\leq\|\tilde v_1-v_1\|=\|\Pi_{V_2}v_1-v_1\|\leq\delta(V_1,V_2)\|v_1\|\leq\frac{\delta(V_1,V_2)}{\sqrt{1-\delta(V_1,V_2)^2}}\|\tilde v_1\|\,,
\end{equation*}
where $v_1\in V_1$ is chosen so that $\tilde v_1=\Pi_{\widetilde V_1}v_1$. Consequently we see that
$$\delta(\widetilde V_1,V_1)\leq \frac{\delta(V_1,V_2)}{\sqrt{1-\delta(V_1,V_2)^2}}\,,$$
and, since $\delta(V_1,\widetilde V_1)=\delta(V_1,V_2)\leq\delta(V_1,V_2)/\sqrt{1-\delta(V_1,V_2)^2}$, inequality \eqref{dxv3} now follows from the left side of \eqref{idete}.

\medbreak

Pick now converging sequences $V_n\to V$ and $W_n\to W$ in $\mathcal G(X)$. The triangle inequality \eqref{triangin} gives
\begin{equation*}
	\delta(V_n,W_n)\leq\delta(V_n,V)+\delta(V,W)+\delta(W,W_n)\leq d(V_n,V)+\delta(V,W)+d(W_n,W),
\end{equation*}
and so, {\em the function $\delta:\mathcal G(X)\times\mathcal G(X)\to\mathbb R$ is upper semicontinuous}. It follows that the set
$$\mathcal U:=\{(V,W):\delta(V,W)<1\}$$
is open in $\mathcal G(X)\times\mathcal G(X)$. In addition we have seen that $\Pi_{W}(V)\subset X$ is closed for every $(V,W)\in\mathcal U$, and so the rule $(V,W)\mapsto\Pi_W(V)$ defines a map $\Pi:\mathcal U\to\mathcal G(X)$. One further has the following:
\begin{lemma}\label{lem911} {$\Pi$ is continuous.}
	\begin{proof}
		To see this we pick some  converging sequence $(V_n,W_n)\to (V,W)$ in $\mathcal U$. In view of \eqref{idete} it suffices to check that
		\begin{equation}\label{eu3}
			\lim_{n\to+\infty}\delta(\Pi_{W_n}(V_n),\Pi_{W}(V))=0=\lim_{n\to+\infty}\delta(\Pi_{W}(V),\Pi_{W_n}(V_n))\,.
		\end{equation}
		Both limits being analogous, let us check the left one. The triangle inequality gives
		$$\delta(\Pi_{W_n}(V_n),\Pi_{W}(V))\leq\delta(\Pi_{W_n}(V_n),\Pi_{W}(V_n))+\delta(\Pi_{W}(V_n),\Pi_{W}(V))\,.$$
		It turns out that both terms in the right hand side converge to zero as $n\to\infty$. Concerning the first one we choose some sequence $\{w_n\}\subset X$ with $w_n=\Pi_{W_n}v_n\in\Pi_{W_n}(V_n)$, $\|w_n\|\leq 1$, and $v_n\in V_n$ for every $n$. Arguing as in \eqref{k312} we see that
		$$1\geq\|\Pi_{W_n}v_n\|=\sqrt{\|v_n\|^2-\|(\id_X-\Pi_{W_n})v_n\|^2}\geq\sqrt{1-\delta(V_n,W_n)^2}\,\|v_n\|\,,$$
		and recalling that $\delta$ is upper semicontinuous we deduce that the sequence $\|v_n\|$ is bounded. Consequently,
		$$\dist(w_n,\Pi_W(V_n))\leq\|w_n-\Pi_W v_n\|=\|\Pi_{W_n}v_n-\Pi_W v_n\|\leq d(W_n,W)\|v_n\|\to 0\,.$$

		\medbreak
		
		Concerning the second one we pick now some sequence $\{\bar w_n\}\subset W$ with $\bar w_n=\Pi_{W}\bar v_n\in\Pi_W(V_n)$ and $\|\bar w_n\|\leq 1$ for every $n$. As before we see that the sequence $\|\bar v_n\|$ should be bounded, and therefore
		$$\dist(\bar w_n,\Pi_W(V))\leq\|\bar w_n-\Pi_W(\Pi_V\bar v_n)\|\leq d(V_n,V)\|\bar v_n\|\to 0\,,$$
		where we have used the equality $\bar w_n-\Pi_W(\Pi_V\bar v_n)=\Pi_W(\Pi_{V_n}\bar v_n-\Pi_V\bar v_n)$. It concludes the argument.
		
	\end{proof} 
\end{lemma}

As a last step before going into the proof of Proposition \ref{corro} we recall a well-known property of relatively compact subsets $A$ of the Hilbert space $X$. It concerns the uniform decay along $A$ of the tails in the series expansions with respect to any orthonormal basis. More precisely, if $\{e_1,e_2,\ldots\}$ is a Hilbert basis of $X$ then
\begin{equation}\label{re}
\lim_{n\to+\infty}\left\|\sum_{j=n+1}^{+\infty}\langle a,e_j\rangle e_j\right\|=0\,,\qquad \text{uniformly with respect to }a\in A\,.
\end{equation}
	\begin{proof}[Proof of Proposition \ref{corro}]  Recalling \eqref{eu12} we pick some Hilbert basis $\{e_1,e_2,\ldots\}$ of $X$ in such a way that all $e_i$'s are eigenvectors of $L$, the first $d:=\dim(\ker L)$ corresponding to the eigenvalue $0$ and the remaining ones associated to some eigenvalue $\lambda_n=\pm 1$. For each $n\geq d$ we denote by $X_n$ the $n$-dimensional, $L$-forward invariant subspace generated by $e_1,\ldots,e_{n}$. The set			$A:=\{K_mx:m\in M,\ \|x\|\leq 1\}$
is relatively compact, and so \eqref{re} gives
\begin{equation}\label{re22}
\|\Pi_{X_n^\bot}\circ K_m\|_{\mathscr L(X)}\leq\sup_{a\in A}\|\Pi_{X_n^\bot}a\|\to 0\qquad\text{as }n\to+\infty,\ \text{unif. w.r.t. }m\in M\,.
\end{equation}
The remaining of the proof will be divided into seven steps:

		\medbreak

		{\em First Step: $\lim_{n\to\infty}\delta(X_{K_m}^0,X_n)=0$, uniformly with respect to $m\in M\,.$}

		\medbreak
		
		To check this assertion we pick some point $m\in M$, some $x\in X_{K_m}^0$ with $\|x\|\leq 1$, and some  $n\geq d$. Then
	\begin{equation*}
	\dist(x,X_n)=\|\Pi_{X_n^\bot}x\|=\|(L\circ\Pi_{X_n^\bot}\circ L)x\|\leq\|(\Pi_{X_n^\bot}\circ L)x\|=\|(\Pi_{X_n^\bot}\circ K_m)x\| \,,
	\end{equation*}
where we have used the equalities $$\Pi_{X_n^\bot}=L^2\circ\Pi_{X_n^\bot}=L\circ\Pi_{X_n^\bot}\circ L,\qquad \|L\|_{\mathscr L(X)}=1,\qquad Lx=-K_mx\,.$$ Consequently $\delta(X_{K_m}^0,X_n)\leq\|\Pi_{X_n^\bot}\circ K_m\|_{\mathscr L(X)}$, and the result follows from \eqref{re22}.

\medbreak

For every $n\geq d$ and $m\in M$ we observe that $X_n=F_{n,m}^0\oplus U_{n,m}$ (orthogonal direct sum), where
$$F_{n,m}^0:=\Pi_{X_n}(X^0_{K_m})\,,\qquad U_{n,m}:=(F_{n,m}^0)^\bot\cap X_n\,.$$
Moreover, the combination of the first step and \eqref{dxv3} gives			
\begin{equation}\label{ex51}
\lim_{n\to\infty}d(X_{K_m}^0,F_{n,m}^0)=0\,,\qquad\text{uniformly with respect to }m\in M\,.
\end{equation}
For each $n\geq d$ and $m\in M$ we consider the linear, selfadjoint endomorphism $T_{n,m}:X_n\to X_n$ defined by 
\begin{equation}\label{Ta1}
T_{n,m}\,x:=\begin{cases}	0,&\text{if }x\in F_{n,m}^0\,,\\
\big[\Pi_{U_{n,m}}\circ(L+K_m)\big]x\,,&\text{if } x\in U_{n,m}\,.
		\end{cases}
		\end{equation}

\medbreak

{\em Second Step: $\lim_{n\to\infty}\left\|(T_{n,m}-(L+K_m))_{|X_n}\right\|_{\mathscr L(X_n,X)}=0$, uniformly with respect to $m\in M$.}

\medbreak

In order to check this new assertion pick $n\geq d$, $m\in M$ and $x\in X_n$ with $\|x\|\leq 1$. Then $T_{n,m}x-(L+K_m)x=\xi+\eta+\zeta$, where
$$\begin{cases}
	\xi=-(L+K_m)(\Pi_{F_{n,m}^0}x)=-(L+K_m)(\Pi_{F_{n,m}^0}-\Pi_{X_{K_m}^0})x\\
\zeta=-\Pi_{F_{n,m}^0}(L+K_m)(\Pi_{U_{n,m}}x)=(\Pi_{X_{K_m}^0}-\Pi_{F_{n,m}^0})(L+K_m)(\Pi_{U_{n,m}}x)\,,\\
\eta=-\Pi_{X_n^\bot}(L+K_m)(\Pi_{U_n,m}x)=-\Pi_{X_n^\bot}K_m(\Pi_{U_{n,m}}x)\,,
\end{cases}$$
All three vectors $\xi,\zeta,\eta$ go to zero as $n\to+\infty$ in a uniform fashion. In the case of $\xi$ and $\zeta$ it follows from \eqref{ex51}, for $\eta$ we apply again \eqref{re22}.

\medbreak

{\em Third Step: There exists some $n_0\geq d$ such that $\ker T_{n,m}=F_{n,m}^0$ for all $n\geq n_0$ and $m\in M$.}

\medbreak

This time we argue by contradiction and assume the existence of sequences 
$$n_k\to+\infty,\qquad m_k\in M,\qquad x_k\in U_{n_k,m_k}\,,$$
with $\|x_k\|=1$ and $T_{n_k,m_k}x_k=0$ for every $k$. From the second step we know that 
$(L+K_{m_k})x_k\to 0$
and from \eqref{ex51} we see that the sequence $y_k:=x_k-\Pi_{X_{K_{m_k}}^0}x_k$ satisfies
$$\|y_k-x_k\|=\big\|(\Pi_{X_{K_{m_k}}^0}-\Pi_{F_{n_k,m_k}^0})x_k\big\|\leq d(X_{K_{m_k}}^0,F_{n_k,m_k}^0)\to 0\qquad\text{ as }k\to+\infty\,.$$
It then follows that 
$$y_k\bot X_{K_{m_k}}^0,\quad\|y_k\|\to 1,\qquad\quad (L+K_{m_k})y_k\to 0\quad\text{as }k\to+\infty\,.$$

By a compactness argument, after possibly passing to a subsequence we may assume $y_k\to y_*$ and $m_k\to m_*$. The continuity of the fiber bundle $X_K^0$ implies that $y_*$ is orthogonal to $X_{K_{m_*}}^0$ and this is not compatible with $$y_*\not=0,\qquad (L+K_{m_*})y_*=0\,.$$

\medbreak

For each $n\geq n_0$ and $m\in M$ we extend $T_{n,m}:X_n\to X_n$ to a linear, continuous, Fredholm and selfadjoint operator $\widehat T_{n,m}:X\to X$ by setting
$$\widehat T_{n,m}x:=\begin{cases}
	T_{n,m}x,&\text{if }x\in X_{n}\,,\\
	Lx+\Pi_{X_n^\bot}K_mx\,,&\text{if }x\in X_n^\bot\,.
\end{cases}$$

{\em Fourth Step: $\widehat T_{n,m}\to L+K_m$ in $\mathscr L(X)$ as $n\to+\infty$, uniformly with respect to $m\in M$.}

\medbreak

In order to check this new assertion we decompose the difference $$\widehat T_{n,m}-(L+K_m)=\big[\widehat T_{n,m}-(L+K_m)\big]\circ\id_X$$ according to the identity
$$\id_X=\Pi_{X_n}+\Pi_{X_n^\bot}\,,$$
and analyze the composition with each of the projections. For the first projection one has:
$$\big[\widehat T_{n,m}-(L+K_m)\big]\circ\Pi_{X_n}= \big[T_{n,m}-(L+K_m)_{|X_n}\big]\circ\Pi_{X_n}\,,$$
and so it goes to zero as $n\to+\infty$ and uniformly with respect to $m\in M$, by the second step. On the other hand, concerning composition with the second projection one has
$$\big[\widehat T_{n,m}-(L+K_m)\big]\circ\Pi_{X_{n}^\bot}=-\Pi_{X_n}\circ K_m\circ\Pi_{X_n^\bot}$$ 
and, consequently,
$$\left\|\big[\widehat T_{n,m}-(L+K_m)\big]\circ\Pi_{X_{n}^\bot}\right\|_{\mathscr L(X)}\leq\| K_m\circ\Pi_{X_n^\bot}\|_{\mathscr L(X)}=\|\Pi_{X_n^\bot}\circ K_m\|_{\mathscr L(X)}\,,$$ 
where we have used that adjoint operators have the same norm and $(K_m\circ \Pi_{X_n^\bot})^*=\Pi_{X_n^\bot}\circ K_m$. The result now follows from \eqref{re22}.

 \medbreak

Going back to the third step we see that for every $n\geq n_0$ and $m\in M$ one has the orthogonal splitting
$$X_n=U_{n,m}^-\oplus F_{n,m}^0\oplus U_{n,m}^+\,,$$
the subspaces $U^\pm_{n,m}$ being $T_{n,m}$-invariant, with $T_{n,m}$ being negative-definite on  $U^-_{n,m}$ and positive-definite on  $U^+_{n,m}$.

\medbreak

On the other hand, the combination of the fourth step and \eqref{eux1} ensures the existence of  constants $0<r_2<R_2$ such that for $n$ big enough, say $n\geq n_1$, one further has
  $$\sigma(\widehat T_{n,m})\subset]-R_2,-r_2[\,\cup\,\{0\}\,\cup\,]r_2,R_2[,\qquad\text{ for every }m\in M\,.$$
 
 Pick some 
 continuously-differentiable, positively-oriented Jordan curve $\gamma^+\subset\mathbb C\backslash\{0\}$ with $]r_2,R_2[\subset\intt(\gamma^+)\cap\mathbb R\subset]0,+\infty[$. Since $\sigma(T_{n,m})\subset\sigma(\widehat T_{n,m})$ we see that $\sigma(T_{n,m})\cap\intt(\gamma^+)=\sigma(T_{n,m})\cap]0,+\infty[$,  and so \eqref{ker123} with $T_{n,m}:X_n\to X_n$ in the place of $L$ gives
 \begin{equation}\label{ker1231}
 {\Pi_{U_{n,m}^+}}_{\big|X_n}=\frac{1}{2\pi i}\int_{\gamma^+}(z\id_{X_n}-T_{n,m})^{-1}dz\,,\qquad n\geq n_1,\ m\in M\,.
 \end{equation}
 
{\em Fifth Step: $\lim_{n\to\infty}\delta(U^\pm_{n,m},X_{K_m}^\pm)=0$, uniformly with respect to $m\in M$.}

\medbreak

We consider for instance the case of $U^+_{n,m}$ and $X_{K_m}^+$, and with this aim we pick $m\in M$, $n\geq n_1$, and $x\in U_{n,m}^+$ with $\|x\|\leq 1$. Combing \eqref{eq12} and \eqref{ker1231} with the equality $\Pi_{U_{n,m}^+}x=x$ we see that
$$\left\|\Pi_{X_{K_m}^+}x-x\right\|\leq\frac{\ell(\gamma_+)}{2\pi}\max_{z\in\gamma^+}\big\|(z\id_{X}-(L+K_m))^{-1}-(z\id_{X}-\widehat T_{n,m})^{-1}\big\|_{\mathscr L(X)}\,,$$
where $\ell(\gamma_+):=\int_0^{2\pi}\|\dot\gamma_+(t)\|dt$ stands for the length of $\gamma_+$.
By the fourth step above, the right hand side goes to $0$ as $n\to+\infty$, uniformly with respect to $m\in M$. It completes the proof of this step.

\medbreak

Combining the fifth step and \eqref{dxv3} we see that the subspaces $Y_{n,m}^\pm:=\Pi_{X^\pm_{K_m}}\big(U^\pm_{n,m}\big)$ satisfy
\begin{equation}\label{xaq1}
	\lim_{n\to\infty}d(U^\pm_{n,m},Y_{n,m}^\pm)=0\,,\qquad\text{uniformly with respect to }m\in M\,.
\end{equation}

\medbreak

{\em Sixth Step: for $n$ big enough, the maps $$Y_n^\pm:M\to\mathcal G(X),\qquad m\mapsto Y_{n,m}^\pm\,,$$ are continuous.}

\medbreak

In order to check this assertion we start by observing that the map 
$$M\to\mathcal G(X)\,,\qquad m\mapsto X^0_{K_m}\,,$$
is continuous (by Proposition 2). On the other hand, the first step ensures that for $n$ big enough one has $\delta(X_{K_m}^0,X_n)<1$ for every $m\in M$. Then Lemma \ref{lem911} implies that the map 
$$M\to\mathcal G(X),\qquad m\mapsto F_{n,m}^0\,,$$
is continuous for $n$ big enough. 
\medbreak

Going back to \eqref{Ta1} we see that
$T_{n,m}=(\id_{X_n}-\Pi_{F_{n,m}^0})\circ(L+K_m)\circ(\id_{X_n}-\Pi_{F_{n,m}^0})$. Thus, for $n$ big enough the map  
$$M\to\mathscr L(X_n),\qquad m\mapsto T_{n,m}\,,$$
is also continuous. Recalling \eqref{ker1231} we see that now that the maps
$$U_n^\pm:M\to\mathcal G(X)\,,\qquad m\mapsto U_{n,m}^\pm$$
are continuous for $n$ big enough. The step follows by combining the fifth step and Lemma \ref{lem911}.

\medbreak

As a consequence of this step we see that the sets $$\mathfrak Y_n^\pm:=\{(m,y):m\in M,\ y\in Y_{n,m}^\pm\}$$ are topological vector subbundles of $\mathfrak X_K^\pm$. Observe also that $X_n=U^-_{n,m}\oplus F^0_{n,m}\oplus U^+_{n,m}$ the direct sum being further orthogonal, and in view of \eqref{ex51}-\eqref{xaq1} and \eqref{oplus} the orthogonal sum $Y_{n,m}:=Y_{n,m}^-\oplus X^0_{K_m}\oplus Y_{n,m}^+$ satisfies
\begin{equation}\label{xaq11}
	\lim_{n\to\infty}d(Y_{n,m},X_n)=0\,,\qquad\text{uniformly with respect to }m\in M\,.
\end{equation}

\medbreak

{\em Seventh Step: the end of the proof.}

\medbreak 

At this moment we recall that the subspaces $X_n$ are $L$-forward invariant, in the sense that $L(X_n)\subset X_n$ for every $n$. Thus, also  $e^{sL}(X_n)\subset X_n$ for all $s\in\mathbb R$, but each operator $e^{sL}$ is a topological isomorphism and so $e^{sL}(X_n)=X_n$. In view of \eqref{xaq11} the subspaces  $Y_{n,m}$ are close to $X_n$ for $n$ big enough. Therefore $e^{sL}[Y_{n,m}]$ is also close to $X_n$ if $n$ is big, and this assertion holds uniformly with respect to $s$ in compact intervals of the real line and $m\in M$. It proves assertion {\em (i)} of the Proposition. Similarly, the remaining second statement follows from \eqref{xaq11} and the fact that the increasing sequence $X_n$ has dense union in $X$. Again we use \eqref{re}. It completes the proof.
\end{proof}
\section{Conflict of interests}
On behalf of all authors, the corresponding author states that there is no conflict of interest.

\medbreak

The authors declare that the data supporting the findings of this study are available within the paper.

\end{document}